\documentclass[11pt,a4paper]{amsart}
\usepackage[T1]{fontenc}
\usepackage[utf8]{inputenc}
\usepackage[includeheadfoot,margin=1in]{geometry}
\usepackage[UKenglish]{babel}
\usepackage{amsaddr}

\usepackage{amsmath,amsthm,amsfonts,amssymb}
\usepackage{textcomp,mathtools}
\usepackage{bm,bbm}
\usepackage[normalem]{ulem}

\usepackage{placeins}  %

\usepackage{standalone} %
\usepackage{xspace} %
\usepackage[final]{microtype}
\usepackage{etoolbox,xkeyval,xpatch}
\usepackage{comment}
\usepackage{enumitem}

\usepackage[dvipsnames]{xcolor}

\usepackage{lmodern}

\usepackage{natbib}

\usepackage{hyperref}
\hypersetup{%
  colorlinks=true, linktocpage=true, pdfstartpage=3, pdfstartview=FitV,%
  breaklinks=true, pageanchor=true,%
  pdfpagemode=UseNone, %
  plainpages=false, bookmarksnumbered, bookmarksopen=true, bookmarksopenlevel=1,%
  hypertexnames=true, pdfhighlight=/O,%
  urlcolor=blue, linkcolor=Maroon, citecolor=ForestGreen, %
}
\usepackage[english=british]{csquotes}
\usepackage[capitalize]{cleveref}

\newtheorem{theorem}{Theorem}%

\newtheorem{proposition}{Proposition}
\crefname{proposition}{Proposition}{Propositions}

\crefname{corollary}{Corollary}{Corollaries}

\newtheorem{lemma}{Lemma}
\crefname{lemma}{Lemma}{Lemmata}

\theoremstyle{definition}
\newtheorem{definition}{Definition}%
\crefname{definition}{Definition}{Definitions}

\crefname{example}{Example}{Examples}

\newtheorem{assumption}{Assumption}
\crefname{assumption}{Assumption}{Assumptions}

\newtheorem{component}{Component}
\crefname{component}{Component}{Components}

\theoremstyle{remark}
\newtheorem{remark}{Remark}
\crefname{remark}{Remark}{Remarks}

\xpatchcmd{\proof}{#1}{\ifstrequal{#1}{\proofname}{\proofname}{\proofname{} #1}}{}{}

\crefname{assumption}{Assumption}{Assumptions}

\newcommand{\xmid}{\mathop{|}}

\makeatletter
\def\lstAZ{A, B, C, D, E, F, G, H, I, J, K, L, M, N, O, P, Q, R, S, T, U, V, W, X, Y, Z}
\def\lstaz{a, b, c, d, e, f, g, h, i, j, k, l, m, n, o, p, q, r, s, t, u, v, w, x, y, z}

\def\lstAZBB{B, C, D, E, F, G, H, I, J, K, L, M, N, O, P, Q, R, T, U, V, W, X, Y, Z}

\newcommand{\MkScr}[1]{\expandafter\def\csname s#1\endcsname{\mathscr{#1}}}
\newcommand{\MkUp}[1]{\expandafter\def\csname u#1\endcsname{\mathrm{#1}}}
\newcommand{\MkFrak}[1]{\expandafter\def\csname f#1\endcsname{\mathfrak{#1}}}
\newcommand{\MkCal}[1]{\expandafter\def\csname c#1\endcsname{\mathcal{#1}}}
\newcommand{\MkBB}[1]{\expandafter\def\csname #1#1\endcsname{\mathbb{#1}}}

\@for\i:=\lstAZ\do{%
	\expandafter\MkScr \i  %
	\expandafter\MkFrak \i  %
	\expandafter\MkUp \i %
	\expandafter\MkCal \i  %
		  }    
\@for\i:=\lstaz\do{%
	\expandafter\MkUp \i   }    
	
\@for\i:=\lstAZBB\do{%
	\expandafter\MkBB \i     }
	    
\makeatother

\DeclareMathOperator{\E}{\mathbb{E}}
\let\Pr\relax
\DeclareMathOperator{\Pr}{\mathbb{P}}
\let\P\Pr

\let\epsilon\varepsilon

\newcommand{\ind}{\mathbf{1}}

\newcommand{\dd}{\mathrm{d}}

\newcommand*{\tran}{{\mkern-1.5mu\mathsf{T}}}

\newcommand{\lf}{{\textsc{LF}}}

\newcommand{\frog}{{\mathrm{FROG}}}

\newcommand{\flip}{{\mathrm{FLIP}}}
\newcommand{\fresh}{{\mathrm{FRESH}}}

\newcommand{\ep}{\epsilon}
\newcommand{\States}{S}
\newcommand{\invo}{{\mathfrak{s}}}

\newcommand{\A}{a}
\newcommand{\B}{b}

\makeatletter
\define@key{todo}{color}[brown]{\def\TodoColor{#1}}
\makeatother

\let\tbl\caption
\makeatletter
\def\figcaptionfont{\reset@font\fontsize{7}{8}\selectfont}

\def\tablecaptionfont{\reset@font\fontsize{11}{13}\selectfont\itshape\leftskip\tableleftskip\rightskip\tablerightskip}%
\def\tablecaptionnumfont{\reset@font\fontsize{11}{13}\selectfont}
\def\tabnotefont{\reset@font\fontsize{9}{11}\selectfont\leftskip\tabledim\rightskip\tabledim}%

\newdimen\figheight
\newdimen\figwidth
\newif\ifFP
\global\FPfalse
\newif\ifCO
\global\COfalse
\renewcommand\thefigure{\@arabic\c@figure}
\def\fps@figure{tbp}
\def\ftype@figure{1}
\def\ext@figure{lof}
\def\fnum@figure{\figurename\ \thefigure}
\def\ftype@affils{1}
\newdimen\tempadimen
\def\@makefigurecaption#1#2{\figcaptionfont%
     \ifdim\figwidth >17.5pc%
        \tempdimen\hsize
    \advance\tempdimen-\figwidth
    \divide\tempdimen by 2
    \tempadimen\figwidth
     \else%
    \tempdimen8.75pc
    \tempadimen17.5pc
  \fi
  \hskip\tempdimen\vbox{\hsize\tempadimen\hskip-.5pt\figcaptionfont\leftskip0pt plus .5fil \rightskip0pt plus -.5fil{#1}\ #2\par\parfillskip0pt plus 1fill}%
}

\def\epsfbox{}
\usepackage{graphics,graphicx}
\newbox\figtempbox
\def\figurebox#1#2#3{%
    \@ifnextchar[{\@figurebox{#1}{#2}{#3}}{\@figurebox{#1}{#2}{#3}[]}}
\def\figuresize#1{\gdef\@figtot{#1}}
\let\@figtot\@empty
\figuresize{1}
\def\@figurebox#1#2#3[#4]{%
      \def\temp{#1}%
      \def\tempa{#2}%
      \ifx\temp\@empty%
      \else%
      \fi%
      \ifx\tempa\@empty%
      \else%
      \fi%
      \gdef\CO{CO}
      \gdef\FP{FP}
      \gdef\@thirdarg{#3}%
      \gdef\@frtharg{#4}%
      \ifx\@frtharg\empty
      \ifx\temp\@empty%
         \global\figheight10pc
      \else%
         \global\figheight#1
      \fi%
      \ifx\tempa\@empty%
         \global\figwidth10pc
      \else%
         \global\figwidth#2
      \fi%
         \ifx\@thirdarg\empty \FPfalse\COfalse\fi
      \else 
         \ifx\@thirdarg\empty \FPfalse\COfalse
         \else
           \ifx\@thirdarg\FP \FPtrue
           \else
             \ifx\@thirdarg\CO \COtrue
             \else
             \fi
           \fi
         \fi  
         \setbox\figtempbox=\hbox{\epsfbox{\ifCO #3-\else\ifFP#3-\fi\fi\ifx#1\@empty#2\@empty\includegraphics{#4}\else\includegraphics[scale=\@figtot]{#4}\fi}}%
         \global\figwidth=\wd\figtempbox
         \global\figheight=\ht\figtempbox
       \fi
       \ifdim\figwidth>14.75pc%
          \ifdim\figwidth<20pc%
            \centerline{\vbox to 0pt{\hsize\figwidth{\figbox}}}%
          \else
            \vbox to 0pt{\centerline{\figbox}}%
          \fi
          \else%
            \vbox to 0pt{\centerline{\figbox}}%
       \fi%
     \vskip-\baselineskip
     \ifodd\c@page
       \vtop to \figheight{\vfill\llap{\hskip0.5pc}\vfill}%
     \else
       \vtop to \figheight{\vfill\llap{\hskip0.5pc}\vfill}%
     \fi
}%

\def\fpofigbox#1{\FPtrue\def\@fpo{#1}}
\def\whiteink{}
\def\blackink{}
\def\figbox{%
     \ifx\@frtharg\empty  
       \noindent\vbox{\hsize\figwidth%
          \hrule\hbox to\figwidth{\vrule\hfill\vbox to\figheight{\hsize\figwidth\vfill}\vrule}\hrule}%
     \else
        \vbox{\vskip.8pt\hsize\figwidth
          \hbox to\figwidth{\vbox to\figheight{\hsize\figwidth\box\figtempbox}}}%
     \fi
     \ifCO\else\ifFP\vbox to 0pt{\vskip-.6\figheight\llap{\hbox to \figwidth{%
             \hfill\blackink\vrule height20pt width220pt depth5pt\whiteink
             \llap{\fontsize{26}{26}\selectfont\bfseries FPO scaled at \@fpo\%}\blackink\hfill}}}%
     \FPfalse\fi\fi%
}
\makeatother

\begin{document}

\title{Rebalancing Markov jump processes for non-reversible continuous-time sampling}

\author[Jansson, Schauer, Seyer, and Sharma]{E.~JANSSON, M.~SCHAUER, R.~SEYER \and A.~SHARMA}
\address{Department of Mathematical Sciences, Chalmers University of Technology and University of Gothenburg, 41296 Gothenburg, Sweden}
\email{rubense@chalmers.se}

\begin{abstract}
Markov chain Monte Carlo methods are central in computational statistics, and typically rely on detailed balance to ensure invariance with respect to a target distribution.
Although straightforward to construct by Metropolization, this can induce diffusion-like exploration of the sample space, requiring careful tuning of parameters such as step size.
We introduce a general mechanism for constructing non-reversible continuous-time samplers, without requiring detailed balance. 
Our approach transforms jump processes satisfying a skew-detailed balance condition for a reference measure into processes sampling a target measure absolutely continuous with respect to it.
Unbounded balancing functions allow such samplers to dynamically select favourable transitions.
We establish invariance under weak criteria and demonstrate how to verify geometric ergodicity.
Numerical experiments demonstrate that the resulting samplers are more robust to parameter tuning.
\end{abstract}

\maketitle
\markboth{E.~Jansson, M.~Schauer, R.~Seyer and A.~Sharma}{Rebalancing Markov jump processes for non-reversible continuous-time sampling}

\section{Introduction}
\label{sec:intro}

Markovian Monte Carlo methods enable sampling from a distribution specified by an unnormalized probability density, for example a posterior distribution in Bayesian inference.
The most common device used to construct such samplers is Metropolization \citep{metropolis_equation_1953,hastings_monte_1970}, the introduction of an accept/reject step in the process to enforce a detailed balance condition with respect to the target distribution.
However, the resulting reversibility of the sampler can introduce random walk behaviour leading to slow mixing, requiring careful tuning of sampler hyperparameters \citep{neal_mcmc_2011,andrieu_peskuntierney_2021}.
This is a driving motivation behind the recent interest in \emph{non-reversible} Monte Carlo methods that potentially overcome these drawbacks, see e.g.\@ \citet{fearnhead_piecewise_2018,power_accelerated_2019,andrieu_peskuntierney_2021,gagnon_theoretical_2024}.

The main contribution of this paper is a rebalancing device that turns a Markov jump process satisfying a skew-detailed balance condition for a reference measure into a sampler for a target measure that is absolutely continuous with respect to the reference measure.
The resulting samplers are not only non-reversible, but achieves the minimal rejection rate in the class of such Markov jump samplers.
This can result in significantly reduced rejection rates compared to their Metropolized counterparts.
Furthermore, the continuous-time construction avoids the computational effort of repeated proposals resulting from repeated rejections,
as well as allowing direct composition of different sampler dynamics.

The main principle underlying the proposed mechanism is skew-detailed balance \citep{power_accelerated_2019,vucelja_liftingnonreversible_2016,turitsyn_irreversible_2011,yaglom_statistical_1949}.
Based on this principle, a general framework for discrete-time non-reversible Markov chain samplers with Metropolization is laid out in \cite{andrieu_peskuntierney_2021}.
Our device instead extends the rebalancing of \citet{zanella_informed_2020} to replace Metropolization. 
A similar approach in a discrete state space setting is taken by \citet{power_accelerated_2019}; see also \cite{hukushima_irreversible_2013} for spin models, \cite{gagnon_nonreversible_2020} for nested model selection, and \cite{schauer_causal_2024} for momentum accelerated sampling on graphs.
We will further review related work in conjunction with the discussion of particular examples of sampler dynamics.
After the first availability of our work, the formalism of the class of Markov jump processes under study in this paper was greatly improved in the reversible case by \cite{livingstone_foundations_2025}, and our work continues the study of the non-reversible case.

\section{Rebalancing device}
\label{sec:rebalancing}
\subsection{Markov jump processes}
First, we introduce the necessary notation and definitions.
Let \((x)^+ = \max\{x,0\}\).
Following \cite{kallenberg_foundations_2002}, let the state space $S$ be a Polish space. Denote by  $\mathcal{B}(S)$ the Borel $\sigma$-algebra on $S$. 
A Markov jump process on $S$ is a continuous-time stochastic process $(Z_t)_{t\geq 0}$ on $S$ with almost surely right-continuous paths that are constant except at isolated jumps.	
We only consider time-homogeneous processes, where $\Pr(Z_t \in B|Z_s = a) = \Pr(Z_{t-s} \in B|Z_0 = a)$ for all $0 \le s \le t$. %
We denote the jump times by $(\tau_n)_{n \in \mathbb{N}}$. The law of the process is determined by
the initial distribution $Z_0 \sim \pi_0$ on $(S, \mathcal B(S))$,
the rate function $\lambda \colon S \to [0,\infty)$, and
the jump kernel $\kappa \colon S \times  \mathcal{B}(S) \to [0,1]$, such that $\Pr(Z_{\tau_{i+1}} \in B \xmid Z_{\tau_{i}} = a) = \kappa(a,B)$ for all $a, B \in S \times \mathcal{B}(S)$.

Conditional on $Z_{\tau_i} = a$, the time until the next jump is exponentially distributed with rate $\lambda(a) = (\E[\tau_{i+1}-\tau_i|Z_{\tau_i} = a])^{-1}$.
This entails, by the Markov property, that $\tau_1/\lambda(Z_0),\ (\tau_{2} - \tau_1)/\lambda(Z_{\tau_1}), \dotsc$ form an independent sequence of $\mathsf{Exp}(1)$ random variables and
$Z_0, Z_{\tau_1}, Z_{\tau_2}, \dotsc$ an embedded discrete-time Markov chain with transition kernel $\kappa$.
We also define the rate kernel $\mu(a ,\,  B) = \lambda(a)\kappa(a ,B)$ as the specific rate of jumps from $a \in \States$ to $B \in \mathcal{B}(S)$.
Both the total rate $\lambda$ and the jump kernel $\kappa$ are determined through
\begin{equation*}
    \lambda(a) = \int_{b \in S} \mu (a , \dd b);\qquad \kappa (a, B) = \frac{\mu(a,B)}{\lambda(a)},\ 
B \in \mathcal{B}(S),\ \lambda(a) \neq 0.
\end{equation*}

\subsection{Skew-detailed balance and rebalancing}\label{sec:device}
The key idea is \emph{skew-detailed balance} \citep{andrieu_peskuntierney_2021,power_accelerated_2019}:
\begin{definition}
We say that a rate kernel $\mu$ is in skew-detailed balance with respect to an involution $\invo$ and a measure $\pi$ if 
\begin{equation}\label{eq:skew-db}
	\int_{S \times S} f(a,b) \pi(\dd a)\mu(a,\dd b) =  \int_{S \times S} f(\invo(a'),\invo(b')) \pi(\dd b') \mu(b',\dd a'),
\end{equation}
for all suitably integrable $f\colon S \times S \to \mathbb{R}$.
\end{definition}
\begin{remark}
    This condition is sometimes equivalently stated as
    \begin{equation*}
	\int_{A \times B} f(a,b) \pi(\dd a)\mu(a,\dd b) =  \int_{\invo (A \times B)} f(\invo(a'),\invo(b')) \pi(\dd b') \mu(b',\dd a'),
    \end{equation*}
for all suitably integrable $f\colon S \times S \to \mathbb{R}$ and all domains \(A \times B \subseteq{} \mathcal B(S \times S)\).
    Crucially, the domain is also transformed by the involution.
    Working with \(S \times S\) simplifies the proofs.
\end{remark}

To construct a sampler for a probability measure $\pi$ on $S$, we require three ingredients:
\begin{component}
A reference measure $\tilde \pi$ on $S$ such that $\pi \ll \tilde \pi$, i.e., $\pi$ is absolutely continuous with respect to $\tilde\pi$.
\end{component}

\begin{component}
An involution $\invo\colon S \to S$, that is, a map satisfying $\invo^2 = \operatorname{id}$, which is isometric under both $\pi$ and $\tilde \pi$, i.e., $\pi(A) = \pi(\invo(A))$ and $\tilde\pi(A) = \tilde\pi(\invo(A))$ for all $A \in \cB(S)$. 
\end{component}

\begin{component}
A rate kernel $\tilde \mu\colon S \times \mathcal{B}(S) \to [0,\infty)$ that is in skew-detailed balance with respect to $\invo$ and $\tilde \pi$.
\end{component}

We then, by a rebalancing device, obtain a Markov jump process \((Z_t)_{t \ge 0}\) with rate kernel \(\mu\) with the target measure \(\pi\) as a stationary measure, that is
\begin{equation}
    \label{eq:stat_def}
    \int_{a \in \States} \P(Z_t \in B \mid Z_s = a) \pi(\dd a) = \pi(B),
\end{equation}
where $ t > s > 0, B \in \mathcal B(S)$.
Under our assumptions, skew-detailed balance together with an additional \emph{semi-local condition} yield the necessary stationarity, by the following theorem proven in the supplement:

\begin{theorem}\label{prop:stationary}
	Let $\pi$ be a probability measure, $\invo$ be a $\pi$-isometric involution, and $\mu$ be a rate kernel with finite expected jump rate $\int_{S \times S} \pi(\dd a)\mu(a, \dd b) < \infty$.
    If a weaker form of \eqref{eq:skew-db} holds,
    \begin{equation}\label{eq:half-skew}
        \int_{S \times S} f(b) \pi(\dd a)\mu(a,\dd b) =  \int_{S \times S} f(\invo(b')) \pi(\dd b') \mu(b',\dd a')
    \end{equation}
    for all bounded measurable $f\colon S \to \mathbb{R}$,
    and the semi-local condition 
	\begin{equation}\label{eq:semi-local}
	\lambda(a) = \lambda(\invo(a)) , \quad a \in \States,
	\end{equation}
	holds with respect to $\invo$ and $\pi$, then there exists a nonexplosive Markov jump process with rate kernel equal to $\mu$ up to a modification on a \(\pi\)-null set.
    This process has $\pi$ as a stationary measure.
\end{theorem}

We now construct our rebalancing device such that it fulfils the above conditions, given in the following theorem proven in the supplement:
\begin{theorem}\label{thm:rebalancing}
	{Let $\pi$ be a probability measure on $S$ that is dominated by a $\sigma$-finite reference measure $\tilde \pi$, $\pi \ll 
	\tilde \pi$, and let $\invo$ be an isometric involution for both.}
	Let $\tilde\mu$ be a rate kernel in skew-detailed balance \eqref{eq:skew-db}  with respect to $\invo$ and $\tilde \pi$, such that $\int_{S\times S} \pi(da)\tilde \mu(a,db) = \int_S \pi(da)\tilde\lambda(a) < \infty$.
	Define
	\begin{equation}
	\lambda_{g}(a,b) =
	\begin{cases}
		g\left(\dfrac{\frac{\dd\pi}{\dd \tilde \pi}(b)}{\frac{\dd\pi}{\dd \tilde \pi}(a)}\right) & \text{if} \quad \frac{\dd\pi}{\dd \tilde \pi}(a) \neq  0, \\ 
		0  & \text{otherwise},
	\end{cases}
	\end{equation}
	where $g\colon [0,\infty)\to [0, \infty)$ is a monotone balancing function satisfying
	\begin{equation}\label{eq:balancing}
	g(0) = 0; \quad g(1) = 1; \quad g(t) = tg(1/t),\ t > 0.
	\end{equation}
	The Markov jump process with rate kernel
	\begin{align}\label{eq:balancedjumpkernel}
		\mu(a, \dd b) = \lambda_{g}(a,b)\tilde{\mu}(a, \dd b)  + \lambda_{\invo}(a) \delta_{\invo(a)}(\dd b),
	\end{align}
	where
	\begin{align*}
	  \lambda_{\invo}(a) =   \left(-\int\lambda_{g}(a,b)\tilde{\mu}(a, \dd b) \right.  + \left.\int \lambda_{g}(\invo(a),b)\tilde{\mu}(\invo(a), \dd b)\right)^+
	\end{align*}
	exists and has stationary measure $\pi$.
\end{theorem}
Examples of balancing functions include, but are not limited to the bounded Metropolis balancing function $g(t) = \min\{1,t\}$ and Barker balancing function $g(t) = 2t/(1+t)$, as well as the unbounded \(g(t) = \sqrt{t}\); but there are many ways of constructing them, see \citet[Remarks~2.1--2.3]{livingstone_foundations_2025}.
As is the case for detailed balance and Metropolization, the claim is only about stationarity and not about ergodicity or even irreducibility; it is possible that \(\tilde\mu\) or even \(\mu\) is transient.
To use the process as a sampler, ergodicity must ultimately be proven.

\begin{remark}
All bounded continuous balancing functions are essentially dominated by \(\min\{1,t\}\) up to a time rescaling factor \citep{andrieu_peskuntierney_2021}.
Boundedness imposes a maximum speed restriction in the tails of the target, and thus unbounded balancing functions are of particular interest.
Although the existence theory is significantly more complicated, they have shown promise in some settings.
In particular, \(g(t) = \sqrt t\) performs well when choosing among different kinds of transitions \citep{zhou_rapid_2022,schauer_causal_2024}.
\end{remark}
\begin{remark}
The semi-local condition \eqref{eq:semi-local} motivates the choice of \(\lambda_\invo\) in \cref{thm:rebalancing} as the minimal solution, similar to the choice in \citet{power_accelerated_2019} along the lines of \citet{turitsyn_irreversible_2011,vucelja_liftingnonreversible_2016}.
Naturally, any solution of the semi-local condition \eqref{eq:semi-local} will yield invariance by the proposition, but our choice is critical to achieve the minimal number of corrections using \(\invo\) necessary to maintain stationarity, which can be seen as a rejection if \(\tilde\mu \circ \invo \) can backtrack; in this sense the rebalanced sampler is rejection-minimal.
It also ensures only one of \(\lambda_\invo(a)\) and \(\lambda_\invo(\invo(a))\) is nonzero, avoiding the introduction of some reversibility from jumps between \(a\) and \(\invo(a)\).
\end{remark}
\begin{remark}
    The trivial involution \(\invo = \mathrm{id}\) always satisfies the hypotheses, but will instead produce a reversible process as \eqref{eq:skew-db} reduces to detailed balance, \(\lambda_\invo \equiv 0\), and one recovers the setting of \citet{livingstone_foundations_2025}.
\end{remark}

\subsection{Construction of rate kernels and compositionality}
A way to construct an appropriate rate kernel $\tilde \mu$ is to use a $\tilde \pi$-isometric deterministic bijective map that is skew-symmetric with respect to $\invo$, in the sense that 
\begin{equation}
\label{eq:Tskew}
T^{-1} = \invo \circ T \circ \invo,
\end{equation}
as shown by the following proposition proven in the supplement:
\begin{proposition}\label{prop:determistic kernel}
	Let $\tilde \pi$ be a measure and $\invo$ a $\tilde\pi$-isometric involution. 
	Suppose $T\colon S \to S$ is $\tilde \pi$-isometric, bijective, and satisfies \eqref{eq:Tskew}.
    Let $\tilde \lambda\colon S \to [0, \infty)$ be a rate function which is invariant under $\invo \circ T$, and define
	\begin{equation*}
	   \tilde\mu(a, \dd b) = \tilde \lambda(a)\delta_{T(a)}(\dd b),
	\end{equation*}
	where $\delta_x$ is the Dirac measure at $x$.
    Then the rate kernel $\tilde \mu$ is in skew-detailed balance with respect to $\invo$ and $\tilde \pi$.
\end{proposition}

Rate kernels can also be superposed to create new dynamics, allowing us to compose new dynamics from simple building blocks.
A remarkable property of the continuous-time construction is that one automatically obtains the correct mixture of the kernels preserving the skew-detailed balance. %
\begin{proposition}\label{lemma:compositionality}
    Let $\tilde \pi$ be a measure and $\invo$ a $\tilde\pi$-isometric involution.
    Suppose the rate kernels \(\tilde\mu_j\), \(j = 1,\dotsc,n\), are in skew-detailed balance with respect to $\invo$ and $\tilde\pi$.
    Then the superposition \(\tilde\mu(a, \dd b) = \sum_{j=1}^n \tilde\mu_j(a, \dd b)\) is in skew-detailed balance with respect to $\invo$ and $\tilde\pi$.
\end{proposition}
\begin{proof}
    Immediate by applying skew-detailed balance termwise.
\end{proof}

\subsection{Implementation}
The samplers obtained by the rebalancing device are all pure jump Markov processes, and so can be simulated in the usual way by a Doob--Gillespie algorithm.
After establishing ergodicity, we then estimate expectations with respect to the target by time averaging along the process trajectory.
However, we can obtain an additional variance reduction by replacing the inter-jump times by their expectations.
This is the final estimator which we use in practice, and this further implies the jump times need not be stored by the implementation, only the rates at each state.
\begin{proposition}
    If the Markov jump process \(Z\) corresponding to the rate kernel \(\mu\) is ergodic, has invariant measure \(\pi\), and \(\lambda > 0\), it holds for any \(\pi\)-integrable function \(f\) that
    \begin{align*}
        \E_{X \sim \pi}[f(X)] %
        &= \lim_{N \to \infty} \left.\left(\sum_{n=0}^N \frac{f(Z_{\tau_n})}{\lambda(Z_{\tau_n})}\right)\middle/\left(\sum_{n=0}^N \frac{1}{\lambda(Z_{\tau_n})}\right)\right.
    \end{align*}
    almost surely, for almost\ every \(Z_0\), where \(\tau_0 = 0\).
\end{proposition}
\begin{proof}
   By hypothesis, \(\pi\) is absolutely continuous with respect to the stationary measure for the jump chain. Thus, the statement follows by a direct application of the ergodic theorem and the Rao--Blackwell theorem under the above integrability conditions, see \cite{cappe_reversible_2003}.
\end{proof}
Geometric ergodicity implies a stronger version of the above proposition where the convergence holds for \emph{every} \(Z_0\) \citep{roberts_harris_2006}.

\section{Creating samplers using the rebalancing device}

We consider a target distribution proportional to \(e^{-U(q)}\dd q\), where \(U\colon \mathbb R^d \to \mathbb R\) is proportional to the negative log density with respect to the Lebesgue measure.
Requiring \(\invo\) to be isometric for this traget is very restrictive in general, so a \emph{lifting} is performed, augmenting the state space with a momentum variable \(p \in \mathbb R^d\) with a prescribed marginal distribution \(\rho(\dd p)\).
This allows freedom to incorporate symmetry in \(\rho\) corresponding to \(\invo\), and the samplers are then constructed with stationary distribution \(\pi(\dd(q,p)) \propto e^{-U(q)}\,\dd q\rho(\dd p)\).
This distribution admits the desired target measure as a marginal, allowing the estimation of expectations from trajectories of the
Markov process.
We construct two samplers, one based on Hamiltonian dynamics and the leapfrog integrator and one based on linear translations and random bounces. 

\begin{figure}[h]
    \figuresize{.66}
    \figurebox{}{\linewidth}{}[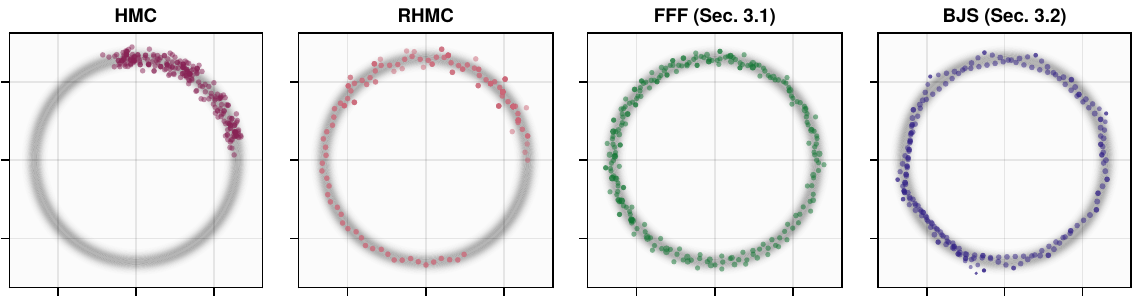]
    \caption{Suboptimal hyperparameters hinders the reversible HMC over short trajectories. Non-reversible methods are more robust, and by reducing rejections, our methods (FFF and BJS) explore further for the same computational cost.}
    \label{fig:donut}
\end{figure}

\subsection{Hamiltonian dynamics}
\label{sec:sampler}
Suppose further that \(\rho(\dd p) \propto e^{-K(p)}\dd p\), where \(K\colon\mathbb R^d \to \mathbb R\).
This suggests sampling the target using Hamiltonian dynamics, where \(U\) is potential energy and \(K\) is kinetic energy, with the stationary distribution described by the Hamiltonian \(H(q,p) = U(q) + K(p)\).
Hamiltonian dynamics are a popular base for Markov chain Monte Carlo methods.
The dynamics are governed by Hamilton's equations, which in this case become
\begin{equation}
\label{eq:hameq}
\dot{q} = \frac{\partial H}{\partial p} = \frac{\partial K}{\partial p},\qquad \dot{p} = -\frac{\partial H}{\partial q} = -\frac{\partial U}{\partial q}.
\end{equation}
The corresponding flow leaves \(H\) constant, and thus is a \(\pi\)-isometry.
Moreover, the dynamics are volume-preserving by Liouville's theorem, suggesting \(\tilde\pi\) as Lebesgue measure, and time-reversible; if the marginal momentum is standard Gaussian, that is \(K(p) = \frac{1}{2} \lVert p \rVert^2\), time-reversibility corresponds to a flip of the momentum sign, so that \(\invo(q,p) = (q, -p)\) satisfies \eqref{eq:Tskew} for the flow.
Unfortunately, unless \(U\) and \(K\) are particularly simple, it is not possible to solve \eqref{eq:hameq} exactly, and one must resort to numerical integration schemes.
The scheme should be \emph{symplectic} to preserve the key properties of the Hamiltonian dynamics, in our case volume preservation and time-reversibility.
The canonical choice is the leapfrog integrator, a second-order method that alternates between updating the momentum and the position variables; in the case of standard Gaussian momentum and a stepsize \(\epsilon\), it becomes
\begin{equation}
\label{eq:leapfrog}
\begin{aligned}
	p_{t + \epsilon/2} &= p_t - \frac{\epsilon}{2} \nabla U(q_t),  \quad
	q_{t + \epsilon} &= q_t + \epsilon p_{t + \epsilon/2},  \quad
	p_{t + \epsilon} &= p_{t + \epsilon/2} - \frac{\epsilon}{2} \nabla U(q_{t + \epsilon}). 
\end{aligned}
\end{equation}
We use the notation $\lf$ to denote the mapping from $(q_t, p_t)$ to $(q_{t + \epsilon}, p_{t + \epsilon})$ given by \eqref{eq:leapfrog}. 
For more details on Hamiltonian dynamics, the leapfrog methods and symplectic integration in general, see e.g.\@{} \cite{leimkuhler_simulating_2004,hairer_geometric_2006,bou-rabee_geometric_2018}.

To obtain our sampler, we select \(\tilde\lambda \equiv 1\) and apply \cref{prop:determistic kernel} with $\lf$, possibly iterated \(L\) times, followed by \cref{thm:rebalancing}.
Nevertheless, the resulting process in general fails to be ergodic.
For this reason, we also introduce refreshments by redrawing the momentum from \(\rho\) at exponential times with constant rate \(\lambda_{\rho} > 0\).
We call the resulting sampler the \emph{Flip-Frog-Fresh sampler} (FFF), given by the rate kernel
\begin{align*}
    \mu_{\mathrm{FFF}}((q,p), \dd(q',p')) &= g\left(e^{-(H(\lf^L(q,p))-H(q,p))}\right)\delta_{\lf^L(q,p)}(\dd(q',p')) \\
    &+ \lambda_\invo(q,p)\delta_{(q,-p)}(\dd(q',p')) + \lambda_\rho\delta_q(\dd q')\rho(\dd p').
\end{align*}

There is a rich literature of related work on extending Hamiltonian Monte Carlo through non-reversibility, continuous-time constructions, or both.
Non-reversibility in samplers based on Hamiltonian dynamics can manifest in numerous ways; some discrete-time examples that rely on skew-detailed balance are
\citet{horowitz_generalized_1991} introducing non-reversibility through partial refreshments and an extra \(\invo\)-flip;
\citet{sohl-dickstein_hamiltonian_2014} selecting between flips and multiple leapfrog steps, although the variant by \citet{campos_extra_2015} shows partial refreshments may reintroduce reversibility;
\citet{fang_compressible_2014} which extends to more variant dynamics;
\citet{thin_nonreversible_2021,andrieu_general_2020,andrieu_monte_2025} which consider more advanced proposal schemes, including using lifted proposals.
These are all instances of the general frameworks in \citet{andrieu_peskuntierney_2021}, making use of the same skew-symmetry of the leapfrog integrator.
Ultimately, working with discrete-time kernels means that some form of Metropolization appears, and there is a correspondence between rebalancing bounded balancing functions and rejections.
However, working in continuous time extends easily to unbounded balancing functions, for which no equivalent discrete chain exists.

Continuous-time versions of HMC have also been previously considered, particularly Randomized Hamiltonian Monte Carlo and numerical variants by \citet{bou-rabee_randomized_2017}, which leads to a very similar sampler that can preserve momentum across leapfrog jumps; however, the corresponding choice of \(\lambda_\invo(q,p) = 1 - \lambda_g(q,p)\) is not minimal, and incurs additional flips compared to the rebalancing device.
As the rebalancing device compares both forward and backward steps, this suggests higher acceptance rates can be achieved due to the oscillatory error in the leapfrog integrator \citep{campos_extra_2015}.

\begin{remark}
    In the special case of visiting or starting close to a mode, the FFF unlike HMC can leave at constant computational cost through the use of holding times.
    As the dimension \(d\) increases, Metropolization requires \(O(d^{1/4})\) rejections to leave a mode following step size scaling results for HMC \citep{neal_mcmc_2011}, yielding a polynomial increase in computational cost.
    Hence, the step size is generally scaled down with dimension to preserve the acceptance probability, but avoiding random walk behavior then requires increasing the number of integrator steps, so the computational cost is still incurred.
\end{remark}

Finally, we establish geometric ergodicity of the sampler under hypotheses on \(U\) comparable to those of \citet{durmus_geometric_2022}:
\begin{theorem}\label{thm:geometricergodicity}
	Assume that the potential $U \in C^{2}(\mathbb{R}^{d})$ satisfies the tail condition that for some $\kappa,r > 0$,  $U(x) - \frac\kappa2 \|x\|^2$ is convex for all $\lVert x \rVert > r$, and that $\nabla U$ is a Lipschitz function with Lipschitz constant $L_{\nabla U}$. 
	If the leapfrog integration step size $\epsilon$ satisfies the bound
	\begin{equation}\label{eq:stepsize}0 < \epsilon  < \frac{ 2 \sqrt{\kappa}}{L_{\nabla U} }
	\end{equation}
	and the balancing function satisfies the following growth condition %
	\begin{equation}\label{eq:FFFgrowth}
        g(t) \le \frac{ct}{\log t},\ t \ge 2,
	\end{equation}
    for some \(c > 0\),
	then FFF is geometrically ergodic (see \citealt{roberts_quantitative_1996} for a definition).
\end{theorem}
The proof is given in the supplement and follows the approach of \citet{meyn_stability_1993-1} as formulated in \citet{roberts_quantitative_1996}.
The continuous-time setup permits a fairly straightforward proof that relies entirely on balancing leapfrog jumps with refreshments.
Hence, it also applies to the aforementioned numerical version of Randomized HMC, answering a conjecture about its geometric ergodicity posed by \citet[Section~6.2]{bou-rabee_randomized_2017}.

\subsection{Bouncy Particle dynamics}\label{subsec_bouncy}
Piecewise deterministic Markov processes have received recent interest as a promising approach to non-reversible continuous-time sampling, see e.g.\@{} \citet{bouchard-cote_bouncy_2018,andrieu_peskuntierney_2021,bierkens_zig-zag_2019,bierkens_methods_2023-1}, which is also based on skew-detailed balance.
However, simulating them in practice is often computationally expensive without tight bounds on gradients of the target due to the use of thinning, and improving this is the subject of ongoing research \citep{bertazzi_approximations_2022-1,andral_automated_2024}.%
Here, we take inspiration from Bouncy Particle sampler (BPS) dynamics \citep{peters_rejection-free_2012, bouchard-cote_bouncy_2018} but use the computationally cheap Markov jump process structure instead, completely avoiding thinning.

We consider a standard Gaussian momentum marginal \(\rho(\dd p)\), although the following generalizes to other elliptically symmetric distributions, and define the reference measure \(\tilde\pi(\dd(q,p)) = \dd q \rho(\dd p)\).
The sampler combines two kernels,
linear steps with scale \(\epsilon\)
\begin{equation}
    \tilde\mu_{\mathrm{step}}((q,p),\dd(q',p')) = \delta_{(q + \epsilon p, p)}(\dd(q',p'))
\end{equation}
and reflections on the isodensity curves of the target potential
\begin{equation}
    \tilde\mu_{\mathrm{refl}}((q,p),\dd(q',p')) = \epsilon \langle p, \nabla U(q)\rangle^+ \delta_{\left(q,~ p - 2\frac{\langle p, \nabla U(q) \rangle}{\lVert \nabla U(q) \rVert^2}\nabla U(q)\right)}(\dd(q',p')),
\end{equation}
both in skew-detailed balance with respect to \(\invo(q,p) = (q,-p)\) and \(\tilde\pi\).
The rate of reflections is inspired by the reflection rate of the BPS, suppressing reflections on trajectories towards a mode where \(U\) is decreasing and vice versa, with the factor of \(\epsilon\) yielding the correct scaling.
Only a directional derivative is required to compute the rate, so the full gradient can be postponed until a reflection actually occurs.
We combine the rate kernels by \cref{lemma:compositionality} and rebalance using \cref{thm:rebalancing}.
Introducing independent refreshments on top, we obtain the \emph{Bouncy Jump sampler} (BJS), given by the rate kernel
\begin{align*}
    \mu_{\mathrm{BJS}}((q,p), \dd(q',p')) &= g\left(e^{-(U(q+\epsilon p)-U(q))}\right)\tilde\mu_{\mathrm{step}}((q,p),\dd(q',p'))
    + \tilde\mu_{\mathrm{refl}}((q,p),\dd(q',p')) \\
    &+ \lambda_\invo(q,p)\delta_{(q,-p)}(\dd(q',p')) + \lambda_\rho\delta_q(\dd q')\rho(\dd p').
\end{align*}

\begin{remark}
    The reflection is volume-preserving, but the choice of \(\tilde\pi\) also allows for stochastic refreshments of \(p\) on reflection as long as the reflection rate is a martingale with respect to the refreshment kernel, which enables ergodicity without separate refreshments \citep{vanetti_piecewise-deterministic_2018,michel_forward_2020}; for brevity we leave this to future work.
\end{remark}

The skew-detailed balance of linear walks have been exploited since \citet{gustafson_guided_1998}.
Augmenting these with reflections has been considered by e.g.\@{} \citet{vanetti_piecewise-deterministic_2018,park_markov_2020,sherlock_discrete_2022}, and is connected to both discretizations of the BPS as well as reflective slice sampling.
Other combinations of steps and reflections were also considered by \citet{ludkin_hug_2023}.
The preceding constructions all relied on Metropolization; here we instead apply the rebalancing device, which crucially allows better control of the rate of \(\invo\)-flips, in particular with the unbounded balancing function \(g(t) = \sqrt t\).
To illustrate this, we introduce two comparison processes: the Rebalanced Guided Walk (BGW, omitting \(\tilde\mu_{\mathrm{refl}}\)) and the Randomized Guided Walk (RGW, with Metropolization) respectively, inspired by \citet{gustafson_guided_1998,bou-rabee_randomized_2017}.
\begin{align*}
    \mu_{\mathrm{BGW}}((q,p), \dd(q',p')) &= g\left(e^{-(U(q+\epsilon p)-U(q))}\right)\tilde\mu_{\mathrm{step}}((q,p),\dd(q',p'))
    \\
    &+ \lambda_\invo(q,p)\delta_{(q,-p)}(\dd(q',p')) + \lambda_\rho\delta_q(\dd q')\rho(\dd p').
\end{align*}
\begin{align*}
    \mu_{\mathrm{RGW}}((q,p), \dd(q',p')) &= \min\left\{1,\middle.e^{-(U(q+\epsilon p)-U(q))}\right\}\tilde\mu_{\mathrm{step}}((q,p),\dd(q',p'))
    \\
    &+ \left(1-\min\left\{1,\middle.e^{-(U(q+\epsilon p)-U(q))}\right\}\right)\delta_{(q,-p)}(\dd(q',p')) + \lambda_\rho\delta_q(\dd q')\rho(\dd p').
\end{align*}
The following proposition proven in the supplement shows the impact of unbounded balancing and bounces on the scaling of the rate of \(\invo\)-flips:
\begin{proposition}\label{prop:bjslikessqrt}
    Assume that the potential \(U \in C^{3}(\mathbb{R}^{d})\).
    For sufficiently small \(\epsilon > 0\), in BJS with \(g(t) = \sqrt{t}\) we have \(\lambda_\invo \sim \cO(\epsilon^3)\), in BJS with \(g(t) = \min\{1,t\}\) we have \(\lambda_\invo \sim \cO(\epsilon^2)\), while BGW with \(g(t) = \min\{1,t\}\) or RGW has \(\lambda_\invo \sim \cO(\epsilon)\).
\end{proposition}

\section{Numerical experiments}
\label{sec:experiments}
We now test the constructed samplers on benchmark problems.
To distinguish the improvements in the rebalancing device from other continuous-time schemes and other forms of non-reversibility in general, we also run classical HMC (discrete-time and reversible, using the package by \citealt{xu_advancedhmcjl_2020}) and numerical RHMC (continuous-time and non-reversible but relies on Metropolized kernels).
The aim is not to provide an exhaustive benchmark of algorithms, but rather to illustrate the preceding theoretical discussion.
Our package implementing the sampler is available as \texttt{FFFSampler.jl}\footnote{Code available at \url{https://github.com/rubenseyer/FFFSampler.jl}}, %
and is compatible with the existing Julia \citep{bezanson_julia_2017} ecosystem of statistical packages, in particular \texttt{AbstractMCMC.jl} and \texttt{Turing.jl}, rendering it accessible for practitioners working within the Julia probabilistic programming framework.

The considered metrics are a subset of the recommendations by \citet{magnusson_posteriordb_2025} applied to each marginal distribution of the parameters.
To measure accuracy, we compare against a reference sample satisfying the criteria in \citet{magnusson_posteriordb_2025}, computing the Wasserstein \(2\)-distance between the reference and the sample empirical distribution.
To measure efficiency, we compute the effective sample size (ESS), which requires sufficient accuracy.
We discretize the trajectory and
use our Julia port of the \texttt{coda} estimator \citep{plummer_coda_2006}.
Although discretization can lose information from short jumps, it allows us to make a direct comparison with discrete HMC.
We then normalize the result by computational effort.
For the Hamiltonian dynamics, we measure the number of evaluations of \(\nabla U\), since using automatic differentiation for \(\nabla U\) the value of \(U\) is obtained simultaneously.
For the Bouncy Particle dynamics, the simpler comparison processes are gradient-free, so we measure in wall clock time since all implementations are by us.

\subsection{Bayesian logistic regression}
We consider a classical example: the posterior of a Bayesian logistic regression model for the German credit data \citep{hofmann_statlog_1994}, where the regression coefficients have zero-mean Gaussian priors with variance \(\sigma^2 = 100\):
\begin{equation*}
    U(\beta) = -\sum_i\left[y_i \langle X_{i*}, \beta\rangle - \log\left(1 + e^{\langle X_{i*}, \beta\rangle}\right)\right] + \frac{\lVert \beta \rVert^2}{2\sigma^2}
\end{equation*}
where \(X\) is a \(1000 \times 25\) design matrix with rows \(X_{i*}\), \(y\) are 0-1 outcomes, and \(\beta\) are the 25 regression coefficients including intercept.
We standardize the covariates to have zero mean and unit variance.
This simple benchmark model has also appeared in e.g.\@{} \citet{hoffman_no-u-turn_2014,durmus_high-dimensional_2019,park_markov_2020}.
The samplers are run starting from \(\beta = 0\) for approximately 250\,000 gradient evaluations (or jumps for the gradient-free methods), testing a grid of hyperparameters, each with 32 replicates.
In this experiment we consider the marginal posteriors of both the parameters and the parameters squared.

For the vast majority of the hyperparameters considered we have good convergence in Wasserstein \(2\)-distance, confirming the validity of the performance estimators.
\Cref{fig:robustness2d-german-ess} shows that 
while good performance for HMC requires tuning for a particular integration time \(L\epsilon\), the three non-reversible alternatives (RHMC, FFF with \(g(t) = \min\{1,t\}\), and FFF with \(g(t) = \sqrt t\)) perform well for a wide range of stepsizes at just \(L = 1\).
This robustness is even displayed by RHMC, although it cannot match the optimal performance of HMC due to excessive rejections at larger step sizes.
In \cref{fig:robustness1d-german} we compare the best slices of the previous figure; the FFF achieves a small but significant improvement of 18\% in the optimal configuration, and the robustness of the non-reversible alternatives clearly manifests as a flatter performance curve.
Moreover, the proportion of flips is considerably lowered by the rebalancing device.
Although the optimal configuration for HMC has a proportion of 0.32, close to the theoretical tuning advice, it is still double than that of the optimal configuration for FFF.
\begin{figure}[hbtp]
    \figuresize{.66}
    \figurebox{}{\linewidth}{}[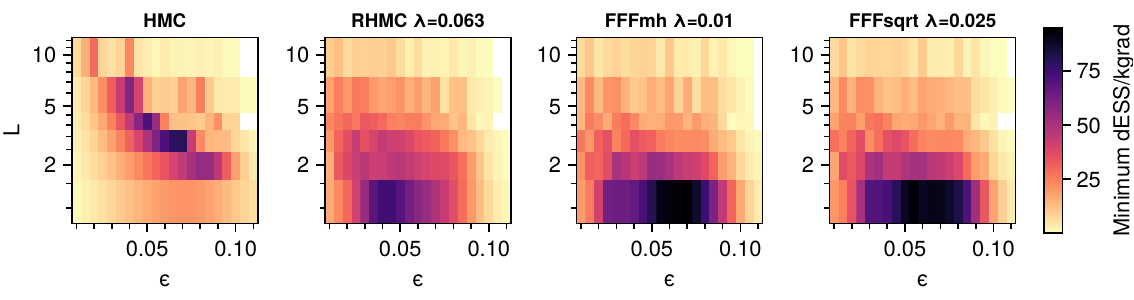]
    \caption{Minimum marginal ESS per \(10^3\) gradient evaluations (higher is better) on Bayesian logistic regression target for varying hyperparameters, averaged across replicates.}
    \label{fig:robustness2d-german-ess}
\end{figure}
\begin{figure}[hbtp]
    \figuresize{.66}
    \figurebox{}{\linewidth}{}[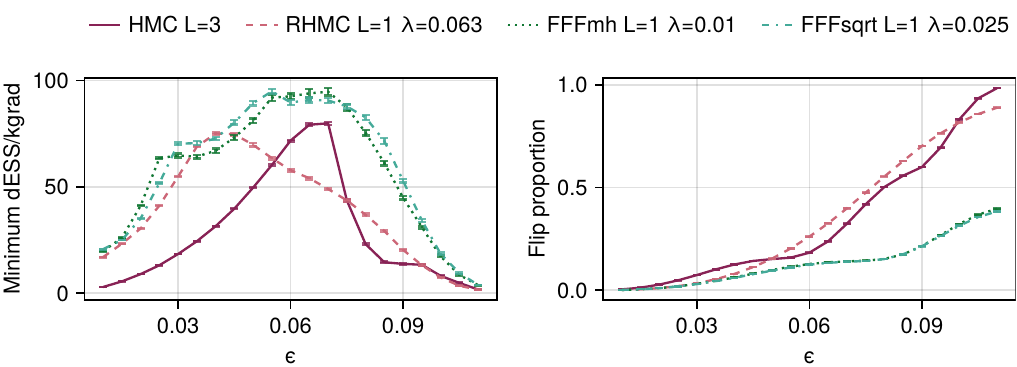]
    \caption{Comparison of minimum marginal ESS per \(10^3\) gradient evaluations (higher is better) and proportion of flips (lower is better) on Bayesian logistic regression target for varying \(\epsilon\) with best \(L\) and \(\lambda_\rho\), averaged across replicates. Error bars show \(\pm 2\) standard errors of the mean.}
    \label{fig:robustness1d-german}
\end{figure}

The linear walk dynamics in the BJS struggle with the varying scales in the posterior.
In \cref{fig:robustness2dbjs-german-ess}, we see that the introduction of bounces greatly improves the performance of the sampler, in particular when using \(g(t) = \sqrt t\), as predicted.
If the underlying dynamics limit exploration, the rebalancing device cannot automatically introduce it; the strength is the relative ease with which we could introduce the additional reflections with the correct scaling.
Nevertheless, \cref{fig:robustness1dbjs-german} shows that BJS does not use the gradient information as efficiently as Hamiltonian dynamics do.
\begin{figure}[hbtp]
    \figuresize{.66}
    \figurebox{}{\linewidth}{}[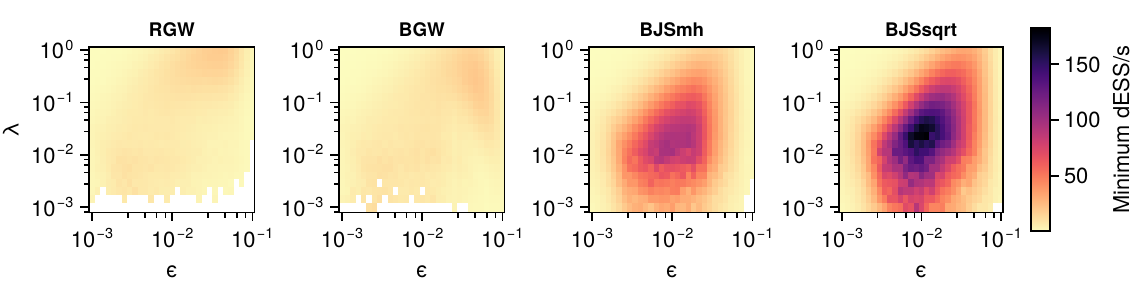]
    \caption{Minimum marginal ESS per wall clock second (higher is better) on Bayesian logistic regression target for varying hyperparameters, averaged across replicates.}
    \label{fig:robustness2dbjs-german-ess}
\end{figure}
\begin{figure}[hbtp]
    \figuresize{.66}
    \figurebox{}{\linewidth}{}[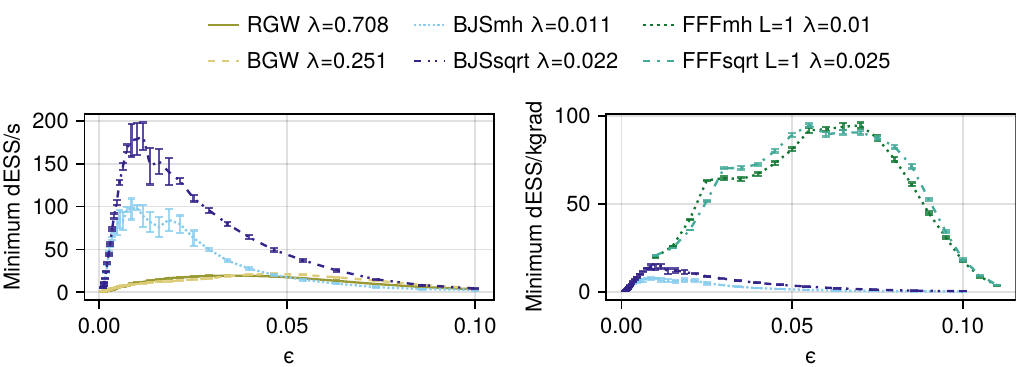]
    \caption{Comparison of minimum marginal ESS per wall clock second (higher is better) and per \(10^3\) gradient evaluations (higher is better) on Bayesian logistic regression target for varying \(\epsilon\) with best \(\lambda_\rho\), averaged across replicates. Error bars show \(\pm 2\) standard errors of the mean.}
    \label{fig:robustness1dbjs-german}
\end{figure}

\subsection{Rosenbrock banana}
We target the Rosenbrock ``banana'' potential
\begin{equation*}
    U(q_1,q_2) = \frac{1}{10}\left(100(q_2 - q_1^2)^2 + (q_1 - 1)^2\right)
\end{equation*}
\citep{goodman_ensemble_2010}, which has the probability mass concentrated on a quadratic ridge, and does not satisfy the assumptions of \cref{thm:geometricergodicity}.
Not only is the target non-isotropic, but the concentrated ridge will exacerbate slow mixing due to reversibility.
The samplers are run starting out in the tail at \([4.678,\, 4.678^2]\) for approximately 3\,500\,000 gradient evaluations, testing a grid of hyperparameters, each with 32 replicates.

We still see in \cref{fig:robustness2d-banana-ess} that a wider range of hyperparameters yield the best performance for the non-reversible samplers, in particular for both FFF variants.
Non-reversibility appears to allow almost an order of magnitude reduction in the number of steps and thus a more efficient sampling kernel.
Comparing both figures, for all samplers seemingly the best estimated efficiency comes with slightly reduced estimated accuracy, but the settings of FFF scoring well in accuracy are still more efficient than the equivalent HMC.
\begin{figure}[hbt]
    \figuresize{.66}
    \figurebox{}{\linewidth}{}[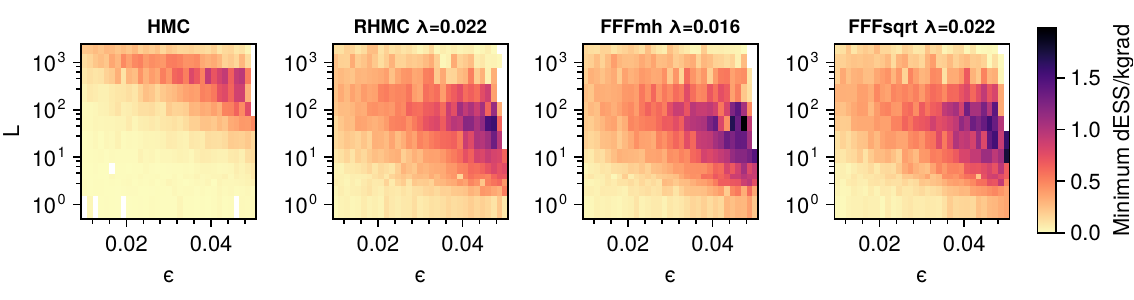]
    \caption{Minimum marginal ESS per \(10^3\) gradient evaluations (higher is better) on Rosenbrock banana target for varying hyperparameters, averaged across replicates.}
    \label{fig:robustness2d-banana-ess}
\end{figure}

\subsection{Bayesian PKPD model}
We target the \texttt{one\_comp\_mm\_elim\_abs} model in PosteriorDB \citep{magnusson_posteriordb_2025}, representing a one-compartment pharmacokinetic model:
\begin{gather*}
    \frac{\dd C}{\dd t} = - \frac{V_m}{V} \frac{C}{K_m + C} + \exp(-k_a t) \frac{D k_a}{V},\quad \hat C(t) \sim \log \mathsf{N}(C(t), \sigma)\\
    k_a, K_m, V_m, \sigma \overset{\text{i.i.d.}}{\sim} \mathsf{Cauchy}^+(0,1),\quad C(0) = 0 
\end{gather*}
where \(C(t)\) is the true concentration (mg/l), \(\hat C(t)\) is the measured concentration (mg/l), \(k_a\) is the dosing rate per day, \(K_M\) is the Michaelis--Menten constant (mg/l), \(V_m\) is the maximum elimination rate per day, \(D\) is the total dosage (mg), and \(V\) is the compartment volume (l).
The data consists of a series of measured concentrations at fixed times together with \(D\) and \(V\).
This ODE-based model has comparatively expensive likelihood evaluations requiring the solution of the differential equation, making it desirable that the number of leapfrog steps are kept low for performance.
Moreover, the constrained parameter space yields non-linear contours in several posterior marginals of pairs, introducing complicated geometry as in the Rosenbrock banana example; this effect is not limited to synthetic examples.
The samplers are run starting at \([0,\, 0,\, 0,\, {-2}]\) for approximately 250\,000 gradient evaluations, testing a grid of hyperparameters, each with 32 replicates.
Evaluation of the likelihood is done through Stan and BridgeStan \citep{stan_development_team_stan_2025,roualdes_bridgestan_2023}.

The conclusion from \cref{fig:robustness2d-posteriordb-ess} is the same as in the previous example: for similar efficiency, one can run shorter integration times requiring fewer steps.
In this case one can even go to very low \(L\) where HMC mixes slowly, which is particularly interesting if embedding the sampler in a wider statistical workflow, as it drastically reduces the number of slow likelihood evaluations per kernel evaluation.
\begin{figure}[hbt]
    \figuresize{.66}
    \figurebox{}{\linewidth}{}[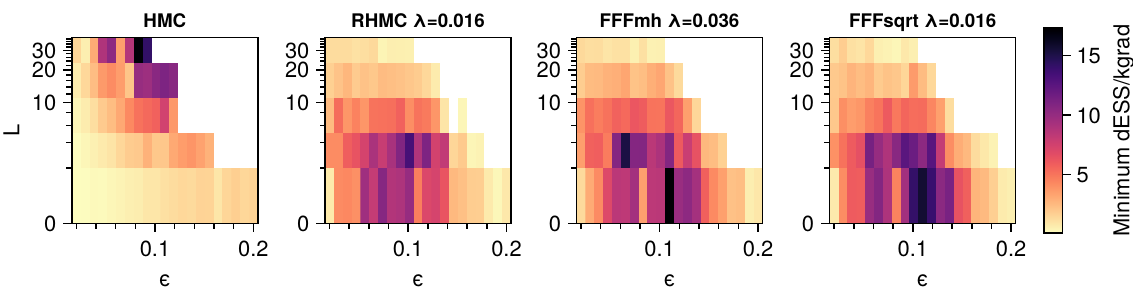]
    \caption{Minimum marginal ESS per \(10^3\) gradient evaluations (higher is better) on Bayesian PKPD target for varying hyperparameters, averaged across replicates.}
    \label{fig:robustness2d-posteriordb-ess}
\end{figure}

\FloatBarrier
\section*{Acknowledgement}
The authors thank Kasper Bågmark, Adrien Corenflos, Klas Modin, and Aila Särkkä for helpful discussions.

\section*{Funding}  %
EJ and AS acknowledge support by the Wallenberg AI, Autonomous Systems and Software Program (WASP) funded by the Knut and
Alice Wallenberg Foundation. 
EJ acknowledges the funding of the Knut and Alice Wallenberg Foundation through grant no. 2024.0440.
RS acknowledges support by foundations managed by The Royal Swedish Academy of Sciences and The Lars Hierta Memorial Foundation.
The numerical experiments were enabled by resources provided by Chalmers e-Commons at Chalmers as well as the National Academic Infrastructure for Supercomputing in Sweden (NAISS), partially funded by the Swedish Research Council through grant agreement no. 2022-06725.

\bibliographystyle{biometrika}
\bibliography{FFF}

\vskip3cm
\setcounter{section}{0}%
\renewcommand\thesection{S\arabic{section}}%
\renewcommand\theHsection{S\arabic{section}}%
\setcounter{theorem}{0}%
\renewcommand\thetheorem{S\arabic{theorem}}%
\renewcommand\theHtheorem{S\arabic{theorem}}%
\setcounter{definition}{0}%
\renewcommand\thedefinition{S\arabic{definition}}%
\renewcommand\theHdefinition{S\arabic{definition}}%
\setcounter{proposition}{0}%
\renewcommand\theproposition{S\arabic{proposition}}%
\renewcommand\theHproposition{S\arabic{proposition}}%
\setcounter{lemma}{0}%
\renewcommand\thelemma{S\arabic{lemma}}%
\renewcommand\theHlemma{S\arabic{lemma}}%
\setcounter{corollary}{0}%
\renewcommand\thecorollary{S\arabic{corollary}}%
\renewcommand\theHcorollary{S\arabic{corollary}}%
\setcounter{equation}{0}%
\renewcommand\theequation{S\arabic{equation}}%
\renewcommand\theHequation{S\arabic{equation}}%

\section{Rebalancing device: Proofs of \cref{sec:rebalancing}}
For a rate kernel \(\mu\) of a Markov jump process, we define the formal operator $\cA$ acting on a bounded measurable function $f \colon S \to \mathbb{R}$ by 
\begin{align}
	\label{eq:generatordefinition}
	\cA f(a) = \int_S (f(b)-f(a)) \mu(a, \dd b).
\end{align}
If the corresponding total rate \(\lambda\) is bounded, by a standard theorem \citep[Proposition~19.2]{kallenberg_foundations_2002} the operator \(\cA\) in fact corresponds to the generator of the jump process, and one could apply Hille-Yoshida theory \citep[Proposition~9.2]{ethier_markov_1986} to establish the invariant measure.
However, more specific results exist that hold even for unbounded \(\lambda\) as long as the jump process with rate kernel \(\mu\) exists for all time (non-explosion).

\begin{proof}[of \cref{prop:stationary}]
We first verify that
    \begin{equation}\label{eq:formalzero}
        \int_S \cA f(a) \pi(\dd a) = \int_{S \times S} (f(b) - f(a)) \mu(a,\,\dd b) \pi(\dd a) = 0
    \end{equation}
    for all bounded measurable test functions $f: S \to \mathbb R$.
	Recall the transformation formula for a bijection \(g: S \to S\):
	\[
	\int _{g^{-1}(A)}f\circ g\; \dd  \nu =\int _{A}f\;\dd (\nu \circ g^{-1}) = \int _{A}f(x) \nu(  g^{-1}(\dd x)).
	\]
	The proposition follows by splitting the integral over the difference,
    which is permitted by $f$ being bounded and $\mu(a, \dd b) \pi(\dd a)$ a finite measure from $\int \pi(\dd a) \mu(a, \dd b) < \infty$, and showing that the first term equals the second.
    Using \eqref{eq:half-skew} and the definition of the rate
	\[\int_{S \times S} f(\B) \mu(\A,\, \dd \B) \pi(\dd \A)
    = \int_{S \times S} f(\invo^{-1}(\B')) \mu(\B' ,\,  \dd \A') \pi(\dd \B')
    = \int_S  \lambda(\B') f(\invo^{-1}(\B'))\pi(\dd \B')\]
	and using \eqref{eq:semi-local} to substitute $\lambda(\B') = \lambda(\invo^{-1}(\B'))$ we obtain
	\[
    \int_S  \lambda(\B') f(\invo^{-1}(\B'))\pi(\dd \B')
	=\int_S \lambda(\invo^{-1}(\B')) f(\invo^{-1}(\B'))\pi(\dd \B')
	=\int_S \lambda(\B') f(\B')(\pi\circ \invo)(\dd \B')
	\]
	where we applied the transformation formula using that $\invo$ is a bijection.
    Now, as $\invo$ is $\pi$-isometric, we have
    \[
    \int_S \lambda(\B') f(\B')(\pi\circ \invo)(\dd \B')
    =\int_S \lambda(\A) f(\A)\pi(\dd \A)
    =\int_{S \times S} f(\A)\mu(\A,\,\dd\B)\pi(\dd \A).
    \]
	This cancels the second term, and we have established \eqref{eq:formalzero} for any bounded measurable test function $f$.

Recall that $\lambda(a) = \mu(a, S)$. Write $\mu(a, \dd b) = \lambda(a)\kappa(a, \dd b)$ with $\kappa(a, \dd b) = \mu(a, \dd b)/\lambda(a)$ if $\lambda(a) > 0$ and $\kappa(a, \dd b) = \delta_a$ otherwise. 
Consider the probability measure $\nu(da) \propto \lambda(a) \pi(da)$.
By \eqref{eq:formalzero},
\begin{equation*}
    \int \nu(\dd a) \kappa(a, \dd b) \propto \int \pi(\dd a)\lambda(a)\kappa(a, \dd b) = \int\pi(\dd a)\mu(a, \dd b) = \lambda(b)\pi(\dd b) \propto \nu(\dd b),
\end{equation*}
and $\nu$ is invariant for $\kappa$.  

Let $(Y_i)_{i \in \NN}$ be a Markov chain with kernel $\kappa$ and start in $Y_1 = a$ under the probability measure $\Pr_a$, $a \in S$; such a chain exists without further assumption by the theorem of Ionescu-Tulcea.  
Choose $A_n = \{\lambda(Y_i) \le n \text{ infinitely often}\}$.
By reverse Fatou's lemma
\[
\lim_{n\to \infty}  \Pr_\nu(A_n) \ge
\lim_{n\to \infty} \limsup_{i} \Pr_\nu(\lambda(Y_i) \le n) = 
\lim_{n\to \infty}  \nu(\lambda \le n) = 1
\]
and by $A_n^\complement \supset A^\complement_{n+1}$,
we have
\[
\Pr_\nu(\bigcap_n A^\complement_{n}) = \lim_n \Pr_\nu(A^\complement_n) = 0 .
\]
Moreover, $0 = \Pr_\nu(\bigcap_n A^\complement_{n}) =\int \Pr_a(\bigcap_n A^\complement_{n}) \nu(\dd a)$, so outside a $\nu$-null set $\mathcal N$ it holds $\Pr_a(\bigcap_n A^\complement_{n}) = 0$. Here $\mathcal N$ is also a $\pi$-null set as $Y$ is absorbed in points with $\lambda=0$.
Then $\pi$-almost surely
\[
\sum_n \frac{1}{\lambda(Y_n)} = \infty.
\]
This now allows to construct a non-explosive process $Z$ with invariant measure $\pi$.
For that, consider the Markov chain $Y'_t = Y_{\min(t, \sigma_{\mathcal N})}$ with $\sigma_{\mathcal N} = \inf\{n \in \NN : Y_n \in \mathcal N\}$.
Let
$\mu'(a, \dd b) = \mu(a, \dd b)\mathbf 1_{\mathcal  N^\complement}(a)$ and define $\lambda'$, $\kappa'$ accordingly.
Then $Y'$ has transition kernel $\kappa'$, still with invariant measure $\nu$. Using $Y'$ in the jump-hold construction, 
explosion $\Pr_a$-almost surely never occurs for any $a \in S$. So a non-explosive process $Z$ with start in $Z_0 = a$ under $\Pr_a$ can be constructed from $Y'$ with \citet[Proposition~12.18]{kallenberg_foundations_2002}.
The process $Z$ has rate kernel $\mu'$ and by  \citet[Proposition~12.23]{kallenberg_foundations_2002} and \eqref{eq:formalzero}  invariant measure $\pi$.
\end{proof}
\begin{remark}
In the preceding theorem, we have obtained almost sure non-explosion from stationarity alone.
With a stronger drift condition \citep[Condition~3.4]{zhang_nonexplosion_2018} one can go further:
If $S$ is Polish, there exists a measurable function $V$ taking values in $[0,\infty)$, a sequence of set $S_n \uparrow S$ such that the rate $\lambda$ is bounded on each $S_n$, $\lim_{n \to \infty} \inf_{x \in S\setminus S_n} V(x) \to \infty$, and there is $\alpha > 0$ such that
\(\mathcal A V(x) \le \alpha V(x)\) holds,
then non-explosion holds for $\mu$ for all points $a \in S$ and $\mu' = \mu$ can be taken in the proof above.
Thus sure non-explosiveness is obtained automatically from the drift conditions used to establish geometric ergodicity.
\end{remark}

To prove \cref{thm:rebalancing}, we require the following lemma:
\begin{lemma}
    If $\pi \ll \tilde \pi$ and $\invo$ is isometric for both, then $\tilde \pi$-almost everywhere
\begin{equation}\label{eq:radon-nikodym}
\frac{\dd \pi}{\dd \tilde \pi}\circ \invo 
= \frac{\dd \pi}{\dd \tilde \pi}.\end{equation}
\end{lemma}
\begin{proof}
	The calculation 
\[
\int_A \frac{\dd \pi}{\dd \tilde \pi}(\invo(a'))  
 \tilde\pi(\dd a')
 = \int_{\invo(A)} \frac{\dd\pi}{\dd \tilde \pi}(a')  
 (\tilde\pi\circ \invo)(\dd a')
 =  \int_{\invo(A)} \frac{\dd \pi}{\dd \tilde \pi}(a')  
 \tilde\pi(\dd a') = \pi(\invo(A)) = \pi(A) 
\]
shows that $ \frac{\dd\pi}{\dd \tilde \pi}\circ  \invo$ is a.e.\@{} a Radon--Nikodym derivative of $\pi$ with respect to $\tilde\pi$.
\end{proof}

\begin{proof}[of \cref{thm:rebalancing}]
		First, $\mu$ is a rate kernel of a Markov jump process
        with finite expected rate under $\pi$: %
\[
\int  \pi(da)\mu(a, db) \le 2\int  \pi(da)\lambda_g(a, b)\tilde\mu(a, db) 
\]
and thus using $g(x) \le \max(1,x)$ (see \citealt[Lemma~2.1]{livingstone_foundations_2025})
\[\le %
2\int_U  \pi(da)\tilde \mu(a, db) 
+
2\int_{U^c}  \tilde\pi(da) {\frac{\dd\pi}{\dd \tilde \pi}(b)}  \tilde \mu(a, db)
\]
where $U = \left\{ \tfrac{\dd\pi}{\dd \tilde \pi}(b)\le\tfrac{\dd \pi}{\dd \tilde \pi}(a)\right\}$ and then
\[
=
2\int_U  \pi(da)\tilde \mu(a, db) 
+
2\int_{U^c}  \tilde\pi(db) {\frac{\dd\pi}{\dd \tilde \pi}(\invo(b))}  \tilde \mu(b, \dd a) = 2 \int  \pi(da)\tilde \mu(a, db), 
\]      
which is finite by assumption (which also justifies using \eqref{eq:skew-db} in the last step).
        Thus 
         by \cref{prop:stationary} it suffices to verify \eqref{eq:half-skew} and \eqref{eq:semi-local} to establish existence and stationarity.
		
		The choice of \(\lambda_\invo\) ensures that \eqref{eq:semi-local} is satisfied by construction.
        To see this, take without loss of generality $a$ with $\lambda_\invo(\invo(a)) = 0$ (otherwise simply apply \(\invo\)).
        It follows that
        \begin{equation*}
            \lambda_\invo(a) = -\int\lambda_{g}(a,b)\tilde{\mu}(a, \dd b) + \int \lambda_{g}(\invo(a),b)\tilde{\mu}(\invo(a), \dd b)
        \end{equation*}
        and we deduce
		\begin{align*}
			\lambda(a) &
		 =\lambda_\invo(a)  + \int_{S}  \lambda_g(a, b)\tilde \mu(a,\,  \dd b)    =  \int_S\lambda_g(\invo(a), b)\tilde \mu( \invo(a) ,\,  \dd  b)   = \lambda(\invo(a)).
		\end{align*}
		
		It remains to show \eqref{eq:half-skew}.
        For later  application of \eqref{eq:skew-db} we note that  
	
    $\tfrac{\dd \pi}{\dd \tilde \pi}(b)\lambda_g(b, a)
		$ is $\tilde\pi(\dd a)\tilde\mu(a,\,\dd b)$-integrable from the finite expected rate as well.
		Define the $\pi$-null set $\cN = \{a \in S\colon \frac{\dd\pi}{\dd \tilde \pi}(a) = 0\}$. For bounded measurable $f\colon S \to \RR$, integrating over $S \times S$,
		\begin{align*}
		\int f(b)  \lambda_g(a,b)\pi(\dd a)\tilde\mu(a,\,  \dd b) 
		& \stackrel{\mathclap{\eqref{eq:balancing}}}{=} 
		\int \ind_{\cN^\complement}(a) f(b) \pi(\dd a)\tilde\mu(a,\,  \dd  b) \dfrac{\frac{\dd\pi}{\dd \tilde \pi}(b)}{\frac{\dd \pi}{\dd \tilde \pi}(a)}\lambda_g(b,a)  \\
		&  = 
		\int    f(b) \tilde\pi(\dd a)\tilde\mu(a,\,  \dd b) \frac{\dd \pi}{\dd \tilde \pi}(b)\lambda_g(b,a) \\
		& \stackrel{\mathclap{\eqref{eq:skew-db}}}=
		\int   f(\invo(b')) \tilde\pi(\dd b') \tilde\mu(b' ,\,  \dd a') \frac{\dd\pi}{\dd \tilde \pi}(\invo(b'))\lambda_g(\invo(b'),\invo(a')) \\
		& \stackrel{\mathclap{\eqref{eq:radon-nikodym}}}=
		\int   f(\invo(b')) \tilde\pi(\dd b') \tilde\mu(b' ,\,  \dd a') \frac{\dd\pi}{\dd \tilde \pi}( b')\lambda_g(b',a') \\
		&  = 
		\int  f(\invo(b')) \pi(\dd b') \lambda_g(b',a')\tilde\mu(b' ,\,  \dd a')  
		.
		\end{align*}
		Therefore,
		\begin{align*}
		\int f( b) \pi(\dd a)\mu(a ,\,  \dd b) &= \int f(b)  \lambda_g(a,b)\pi(\dd a)\tilde\mu(a,\,  \dd b) + \int \lambda_{\invo}(a) f(\invo(a)) \pi(\dd a) \\
		&=
		\int  f(\invo(b')) \pi(\dd b') \lambda_g(b',a')\tilde\mu(b' ,\,  \dd a') +  \int \lambda_{\invo}(b') f(\invo(b')) \pi(\dd b')
		\\
		&  = 
		\int  f(\invo(b')) \pi(\dd b') \mu(b' ,\,  \dd a').
		\end{align*}
		 The proof is complete.
\end{proof}

\begin{proof}[of \cref{prop:determistic kernel}]
    For any $\tilde \pi \tilde\mu$-integrable $f\colon S \times S \to \RR$, a direct calculation with the change of variables \(a \mapsto (\invo \circ T)(b')\) yields
    \begin{align*}
        \int_{S \times S} f(a,b)\,\tilde \pi(\dd a)\,\tilde{\mu}(a,\dd b)
        &= \int_{S} f(a,T(a))\tilde \lambda(a)\,\tilde \pi(\dd a) \\
        &= \int_S f((\invo \circ T)(b'), (T \circ \invo \circ T)(b')) \tilde\lambda((\invo \circ T)(b'))\,\tilde\pi(\dd b') \\
        &= \int_S f((\invo \circ T)(b'), \invo(b')) \tilde\lambda(b')\,\tilde\pi(\dd b') \\
        &= \int_{S\times S} f(\invo(a'), \invo(b'))\,\tilde\pi(\dd b')\tilde\mu(b',\dd a')
    \end{align*}
    as required.
\end{proof}

One appealing transformation of the kernel is the replacement of \(\invo\)-events by some other \(\pi\)-isometric involution \(\mathfrak t\) that is more conducive to exploration, but does not satisfy skew-detailed balance.
Such transformations have appeared in e.g.\@{} \citet{sherlock_discrete_2022,vanetti_piecewise-deterministic_2018}, and is more broadly linked to the connection between lifting and delayed rejection \citep{andrieu_peskuntierney_2021}.
The following proposition shows this extension is also applicable in the rebalancing setting:
\begin{proposition}
    The conclusion of \cref{thm:rebalancing} holds even if instead of \eqref{eq:balancedjumpkernel}
    we define
    \[
        \mu(a, \dd b) = \lambda_{g}(a,b)\tilde{\mu}(a, \dd b)  + \min\left(\lambda_{\invo}(a),\lambda_{\invo}(\mathfrak{t}(a))\right)\delta_{\mathfrak{t}\circ\invo(a)}(\dd b) +  \left(\lambda_{\invo}(a) -\lambda_{\invo}(\mathfrak{t}(a))\right)^+\delta_{\invo(a)}(\dd b)
    \]
    for some other $\pi$-isometric involution $\mathfrak t$ that \emph{commutes} with $\invo$.
\end{proposition}
\begin{proof}
By $\max\{x-y,0\} + \min\{x, y\} = x$, the total jump rate is unchanged, so \eqref{eq:semi-local} holds.
Moreover,
\begin{align*}
&  \int \min\left(\lambda_{\invo}(a),\lambda_{\invo}(\mathfrak{t}(a))\right) f(\mathfrak{t}\circ\invo(a)) \pi(\dd a) + 
 \int \left(\lambda_{\invo}(a) -\lambda_{\invo}(\mathfrak{t}(a))\right)^+ f(\invo(a)) \pi(\dd a) 
\\
&=\int \min\left(\lambda_{\invo}(  \mathfrak{t}(a)),\lambda_{\invo}(  \mathfrak{t}\circ\mathfrak{t}(a))\right) f(  \mathfrak{t}\circ\invo\circ\mathfrak{t}(a)) \pi(\dd a) + 
 \int \left(\lambda_{\invo}(a) -\lambda_{\invo}(\mathfrak{t}(a))\right)^+ f(\invo(a)) \pi(\dd a) 
\\
&=\int \min\left(\lambda_{\invo}(  \mathfrak{t}(a)),\lambda_{\invo}(   a)\right) f(\invo(a)) \pi(\dd a) + 
 \int \left(\lambda_{\invo}(a) -\lambda_{\invo}(\mathfrak{t}(a))\right)^+ f(\invo(a)) \pi(\dd a) \\
&=
  \int \lambda_{\invo}(a') f(\invo(a')) \pi(\dd a')
\end{align*}
thus showing \eqref{eq:half-skew}.
Hence the conditions for \cref{prop:stationary} are still satisfied.
\end{proof}

\section{Geometric ergodicity of FFF: Proof of \cref{thm:geometricergodicity}}\label{app:geometricergodicity}
We establish the geometric ergodicity of the proposed sampler for a class of potentials \(U\), using the Meyn--Tweedie approach \cite{meyn_stability_1993-1} as formulated in \cite{roberts_quantitative_1996}; this formulation allows a particularly straightforward proof.
The Meyn--Tweedie approach has also been used for HMC in e.g. \cite{durmus_irreducibility_2020-1,livingstone_geometric_2019}.
Geometric ergodicity shows not only that refreshments are enough for convergence of the ergodic averages, but also, thanks to the exponentially fast convergence, entails a corresponding central limit theorem.
The proof is divided into two parts:
First, we prove a drift condition, i.e., that there is a suitable Lyapunov function $V\colon \mathbb{R}^d \times \mathbb{R}^d \to [1,\infty)$ for which there are constants $\alpha > 0$ and $\beta \in \RR$ such that 
\begin{equation}
\label{eq:drift}
\mathcal A V \le -\alpha V + \beta. 
\end{equation}
Then, we prove a minorization condition, meaning that there exists a probability measure $\nu$ on $\RR^d\times \RR^d$, a time \(t^* > 0\), and a constant $\varsigma \in (0,1)$ such that
\begin{equation}
\label{eq:minorization}
\Pr(Z_{t^*} \in B \xmid Z_0 = x) \geq \varsigma \nu(B),\ x \in G, B \in \mathcal{B}(\mathbb{R}^{2d}),
\end{equation}
where $G = \{x \in \RR^d\times \RR^d: V(x) \leq R\}$ for some $R \ge 2\beta/\alpha$.

We recall the hypotheses on \(U\) in \cref{thm:geometricergodicity}:
\begin{assumption}\label{ass:unlocallyconvex}
	The potential $U \in C^{2}(\mathbb{R}^{d})$ satisfies the tail condition that for some $\kappa,r > 0$,  $U(x) - \frac\kappa2 \|x\|^2$ is convex for all $\lVert x \rVert > r$, and $\nabla U$ is a Lipschitz function with Lipschitz constant $L_{\nabla U}$. 
\end{assumption}

\subsection{Drift condition}
\begin{proposition}\label{prop:drift}
    Let \cref{ass:unlocallyconvex}, \eqref{eq:stepsize} and \eqref{eq:FFFgrowth} hold.
    Then the function
    \begin{equation}\label{eq:V}
        V(q,p) = C + H(q,p)
    \end{equation}
    where \(C \in \mathbb R\) is such that \(\inf V = 1\) satisfies \eqref{eq:drift} for some \(\alpha > 0\) and \(\beta \in \mathbb R\).
\end{proposition}
Our proof of the drift condition requires two lemmas.
We first give a sufficiently fine-grained Taylor expansion of the terms for how the Hamiltonian changes for a forward step. 
\begin{lemma}\label{lemma:taylor}
    Let $U$ be twice continuously differentiable and
    $(q_1, p_1) = \lf(q, p)$. Then with ${\delta} = q_1 - q =   \epsilon p - \frac{\epsilon^{2}}{2} \nabla U(q)$,
 \begin{align}\label{eq:hamdiff1}
 H(q_1, p_1) - H(q, p)  &=
  \frac\epsilon2 \langle p , R_2\rangle  + \frac{\epsilon^{2}}{8}\|R_2\|^{2} 
   + \langle {\delta}  , R_1 - R_2\rangle,
   \end{align}
where
$
R_1 =  \int_0^1  \nabla^{2}U(q + \alpha  {\delta}) {\delta} (1-\alpha)  \dd \alpha$ and $ 
R_2 =  \nabla U(q_1) - \nabla U(q)
$.
\end{lemma}
\begin{proof}
We have
\begin{align*}%
    q_1 &= q + \ep p - \frac{\ep^2}{2}\nabla U(q)  = q + \delta,\\
    p_1 &= p - \frac{\ep}{2}\big(\nabla U(q)  + \nabla U(q_1)\big) = 
    p - \epsilon \nabla U(q) - \frac{\epsilon}{2} R_2.
\end{align*}

Using Taylor's formula, 
we get
\begin{align}
    \label{eq:taylor1}
    U(q_{1}) - U(q)   &
   =   \langle{\delta} , \nabla U(q) + R_1 \rangle .
\end{align}

Expanding $\|p_1\|^{2}/2$ leads to
\begin{align*}
    \frac{\|p_1\|^{2}}2 &= \frac12 \Big\|p - \frac{\epsilon}{2}(\nabla U(q)  + \nabla U(q_{1}))\Big\|^{2}  
    = \frac12 \Big\|p - \epsilon\nabla U(q)  - \frac\epsilon2 R_2\Big\|^{2} 
    \\ & 
    = \frac{\|p\|^{2}}2 + \frac{\epsilon^{2}}{2}\|\nabla U(q)\|^2 + \frac{\epsilon^{2}}{8}\|R_2\|^{2} -  \epsilon \langle p, \nabla U(q)\rangle - \frac\epsilon2 \langle p,  R_2\rangle + \frac{\epsilon^2}2 \langle  \nabla U(q) , R_2\rangle.
\end{align*}

By identifying terms corresponding to $\delta$ we rewrite to
\begin{align}
\label{eq:onehand}
\frac{\|p_1\|^{2}}2 - \frac{\|p\|^{2}}2
&= -   \langle {\delta}, \nabla U(q) + R_2\rangle  + \frac \epsilon2 \langle p , R_2\rangle  + \frac{\epsilon^{2}}{8}\|R_2\|^{2}.   %
\end{align}
Since
\begin{align*}
    H(q_1,p_1)-H(q,p) = \frac{\|p_1\|^{2}}2-\frac{\|p\|^{2}}2 + U(q_1)-U(q),
\end{align*}
we have that \eqref{eq:hamdiff1} follows by \eqref{eq:onehand} and \eqref{eq:taylor1}.
\end{proof}

Next, we require control of the remainder in our expansion. 
\begin{lemma}\label{lemma:remainder1}
Let \cref{ass:unlocallyconvex} hold. Then, in the setting of the preceding \cref{lemma:taylor}, there exists a $C_\kappa \ge 0$ such that
\[
\langle {\delta}, R_1 - R_2 \rangle \le -\frac12 \kappa \|{\delta}\|^2 + C_\kappa\|{\delta}\|,
\]
where  $ \delta = \epsilon p - \frac{\epsilon^{2}}{2} \nabla U(q)$. 
\end{lemma}
\begin{proof}
By assumption, the function $h = U(x) - \kappa \|x\|^2/2$ is convex outside a ball with radius $r \ge 0$. 
Let
$\ell \in \RR$ denote the smallest (possibly negative) eigenvalue of $ \nabla^2 h$ inside this ball.

With $R_2 = \int_0^1  \nabla^2 U(q + \alpha {\delta}){\delta} \dd \alpha$, we expand
\begin{align*}
\langle {\delta}, R_1 - R_2 \rangle 
&= -\int_0^1 \kappa {\delta}^\tran  {\delta} \alpha \dd \alpha  -\int_0^1 {\delta}^\tran \nabla^2 h(q + \alpha {\delta}){\delta} \alpha \dd \alpha \\
&\le -\frac12 \kappa\|{\delta}\|^2 
+\int_0^{\|{\delta}\|} \left(-{\delta}^\tran \nabla^2 h\Big(q + u\frac{ {\delta}}{\|{\delta}\|}\Big){\delta}\right)^+ \frac{u}{\|{\delta}\|^2} \dd u,
\end{align*}
where we have  substituted $u = \alpha\|{\delta}\|$, so that instead of integrating along a line of length $1$, we integrate along a line of length $\|{\delta}\|$.
In this case, only the part of that line overlapping with the ball of radius $r$ contributes, and we bound using $u \le \|{\delta}\|$ to obtain
\[
\langle {\delta}, R_1 - R_2 \rangle \le -\frac12 \kappa\|{\delta}\|^2 +
2 r\max(0,-\ell)\|{\delta}\|
\]
and we are done with \(C_\kappa = 2 r \max(0, -\ell)\).
\end{proof}
We now proceed with the proof of the drift condition.
\begin{proof}[of \cref{prop:drift}]
Two direct implications of \cref{ass:unlocallyconvex} are that 
\begin{equation}\label{eq:cdef}
\inf_{q \in \RR^d} \{U(q) + C\} = 0
\end{equation}
for some $C \in \RR$, and that there exists a constant $\tilde{C} = \sup_{\lVert q \rVert \le r} \sqrt{2\kappa (U(q) + C)} \geq 0$ such that we have the bound
\begin{equation}\label{eq:gradbound}
	\|\nabla U(q) \| \ge \sqrt{2\kappa (U(q) + C)} - \tilde{C}.
\end{equation}
The latter follows as one may apply a standard convexity inequality \cite[4.12]{bottou_optimization_2018} but make the statement trivial inside the ball where convexity does not hold.

We decompose the generator $\mathcal{A}$ into the three terms $\mathcal{A}_{\frog}$, $\mathcal{A}_{\flip}$, and $\mathcal{A}_{\fresh}$ corresponding to the three event types. We have %
\begin{align*}
\mathcal{A}_\fresh V(q, p) 
&= \lambda_\rho \left(d/2 -  \|p\|^2/2  \right),
\\
\cA_\flip V(q, p) &= 0,
\\
\cA_\frog V(q, p) &= g\left(e^{-(H(q_1, p_1) - H(q, p))}\right)\left[H(q_1, p_1) - H(q, p)\right],
\end{align*}
where $(q_1,p_1) = \lf(q,p)$.
Let
\begin{equation*}
{\delta} =  q_1 - q = \epsilon p - \frac{\epsilon^{2}}{2} \nabla U(q).
\end{equation*}
We divide our analysis into two cases depending on whether \(\lVert p \rVert\) is ``small'' or ``large'', by considering the set
\begin{equation*}
    A_1 = \{q, p\colon \|p\| \le K_1   \sqrt{|U(q) + C|}\}
\end{equation*}
where $C$ is given by \eqref{eq:cdef} and \(K_1 > 0\) is chosen later in the proof.
In the first case, $(q,p) \in A_1$; in the second case, $(q,p) \in A_2 = A_1^\complement$.

\emph{Case \(A_1\).}  %
Using \cref{lemma:taylor}, we have
\[ H(q_1, p_1) - H(q, p)  =
  \frac\epsilon2 \langle p , R_2\rangle  + \frac{\epsilon^{2}}{8}\|R_2\|^{2} 
    + \langle \delta  , R_1 - R_2\rangle,
\]
where
$
R_1 =  \int_0^1  \nabla^{2}U(q + \alpha  {\delta}) {\delta} (1-\alpha)  \dd \alpha,\; 
R_2 =  \nabla U(q_1) - \nabla U(q)
$.
It now follows by elementary inequalities that
\begin{align*}
    \langle p, R_2 \rangle &\le \lVert p \rVert \lVert R_2 \rVert, &
    \| R_2\| & \leq  L_{\nabla U}\|q_1 - q \|  =  L_{\nabla U}\|\delta\|,
\end{align*}
and furthermore by \cref{lemma:remainder1} that
\begin{align*}
    \langle \delta, R_1 - R_2 \rangle &\leq  -\frac12 \kappa \|{\delta}\|^2 + C_\kappa\|{\delta}\|.
\end{align*}
Denoting $-c := \frac{\epsilon^{2}}{8}L_{\nabla U}^2 - \frac12 \kappa$, we have the following bound:
 \begin{align}
 H(q_1, p_1) - H(q, p) &
 \leq   \frac\epsilon2   \| p \|  \| R_2 \|  + \frac{\epsilon^{2}}{8}\|R_2\|^{2} 
    + \langle \delta, R_1 - R_2 \rangle \nonumber\\
  & \le   \frac\epsilon2   L_{\nabla U}\|p\| \| \delta \|  + \frac{\epsilon^{2}}{8} L_{\nabla U}^2  \| \delta \|^2  -\frac12 \kappa \|{\delta}\|^2 + C_\kappa\|{\delta}\| \nonumber\\
 &\le
\frac\epsilon2   L_{\nabla U}\|p\|  \,\|\delta\|+ \left(\frac{\epsilon^{2}}{8}L_{\nabla U}^2 - \frac12 \kappa\right)\|\delta \|^2   + C_\kappa\|\delta\| \nonumber\\
  &\le
\underbrace{\frac\epsilon2   L_{\nabla U}\|p\|  \,\|\delta\|- \frac{c}2\|\delta \|^2 }_{\text{I}}   \underbrace{-\,\frac{c}2\|\delta \|^2  + C_\kappa\|\delta\|}_{\text{II}}, \label{fff_eqn_5.5}
\end{align}
where $c > 0$ due to \eqref{eq:stepsize}.
The bound in \eqref{eq:gradbound} leads to 
\begin{align}
\|\delta\| \ge   \frac{\ep^2}{2}\|\nabla U(q)\| - \ep\|p\|  
\ge \frac{\ep^2\sqrt{2\kappa}}{2}\sqrt{ U(q) + C } - \ep\|p\| - \frac{\ep^2}{2}\tilde{C} .  \label{fff_eqn_5.6}
\end{align}
 We first discuss the term (I) in \eqref{fff_eqn_5.5}.
 Since $(q,p) \in A_1$, therefore
 \begin{align*}
 \| \delta \| + \frac{\ep^2}{2}\tilde{C} \geq \bigg( \frac{\ep^2 \sqrt{2 \kappa}}{2 K_1} - \ep\bigg)\| p\|. 
 \end{align*}
Take $K_1$ to satisfy the following:
\begin{align}
     \frac{\ep \sqrt{2 \kappa}}{2 K_1} =  \frac{ L_{\nabla U}}{c } + 1 \iff
    K_1 = \frac{c}{2(L_{\nabla U}+c)}{\epsilon \sqrt{2 \kappa}}. \label{fff_k_1}
\end{align}
Therefore, we get
\begin{align}\label{fff_eqn_A10}
     \frac{\ep L_{\nabla U}}{2}\| p\| \leq  \frac{c}{2}\bigg(\| \delta \| + \frac{\ep^2}{2}\tilde{C}\bigg).
\end{align}
To deal with term (II) in \eqref{fff_eqn_5.5}, we use Young's inequality to arrive at
\begin{align}
    C_{k}\|\delta\| \leq \frac{c}{4}\|\delta\|^2 + \frac{1}{ c }C^{2}_{\kappa}.  \label{fff_eqn_A11}
\end{align}
Using \eqref{fff_eqn_A10} and \eqref{fff_eqn_A11} in \eqref{fff_eqn_5.5}, we obtain
\begin{align}
    H(q_1, p_1) - H(q,p) \leq -\frac{c}{4}\| \delta\|^2 +  \frac{c \ep^2}{4}\tilde{C} + \frac{1}{ c}C^{2}_{\kappa}. \label{fff_eqn_S17}
\end{align}
Consequently, we obtain
\begin{align}
\cA V(q, p) &\le g\big(\exp(- H(q_1, p_1) - H(q,p))\big) \bigg(  -\frac{c}{4}\| \delta\|^2 +  \frac{c \ep^2}{4}\tilde{C} + \frac{1}{ c}C^{2}_{\kappa} \bigg)\nonumber  \\   &  \quad -\frac{\lambda_\rho}{2}    \|p\|^2 + \frac{d}{2}. \label{fff_eqn_S18}
\end{align}
From \eqref{fff_eqn_5.6}, we have
\begin{align*}
    \|\delta\|   
\ge \bigg(\frac{\ep^2\sqrt{2\kappa}}{2} - \ep K_1 \bigg)\sqrt{ U(q) + C } - \frac{\ep^2}{2}\tilde{C}. 
\end{align*}
The choice of $K_1$ in \eqref{fff_k_1} yields
\begin{align*}
    \|\delta\|   
\ge \frac{\ep^2\sqrt{2\kappa}}{2} \frac{L_{\nabla U}}{L_{\nabla U} +c}\sqrt{ U(q) + C } - \frac{\ep^2}{2}\tilde{C}. 
\end{align*}
This implies 
\begin{align*}
    \bigg(\|\delta\|  +  \frac{\ep^2}{2}\tilde{C}\bigg)^2 \geq K_2 (U(q) + C)
\end{align*}
with $K_2 := \Big(\frac{\ep^2\sqrt{2\kappa}}{2} \frac{L_{\nabla U}}{L_{\nabla U} +c}\Big)^2 $.  Using an elementary inequality, we have
\begin{align*}
    2 \bigg(\|\delta\|^2  +  \frac{\ep^4}{4}\tilde{C}^2\bigg) \geq \bigg(\|\delta\|  +  \frac{\ep^2}{2}\tilde{C}\bigg)^2 \geq K_2 (U(q) + C ).
\end{align*}
Therefore, 
\begin{align}
    \|\delta\|^2 \geq \frac{K_2}{2} U(q)  +\frac{1}{2} K_2 C -\frac{\ep^4}{4}\tilde{C}^2 . \label{fff_A_16}
\end{align}
Further, \eqref{fff_eqn_S17} implies  there is an upper bound $\bar{H} = \frac{c \ep^2}{4}\tilde{C} + \frac{1}{ c}C^{2}_{\kappa}$ uniform on $A_1$, i.e.,  
\begin{align}
    H(q_1, p_1) - H(q,p) \leq \bar{H}, \label{fff_eqn_S20}
\end{align}
for all $(q, p) \in A_1$. Using properties of the balancing function $g$, we have 
\begin{equation}\label{fff_eqn_S21}
g(t) \geq \min\{1 , t\}.  
\end{equation}
Thus on $A_1$, using \eqref{fff_A_16}-\eqref{fff_eqn_S21} in \eqref{fff_eqn_S18}, we get
\[
\mathcal A V \le -\alpha_1 V + \beta_1, 
\] 
for some $\alpha_1 > 0, \beta_1 \in \RR$.

\emph{Case \(A_2\).} We now consider $A_2 = A_1^\complement$, where $\|p\|$ is not small relative to $\sqrt{ U(q) + C}$. 

By \eqref{eq:balancing} followed by \eqref{eq:FFFgrowth}, for \(e^x > 2\)
\[
    x g(e^{-x}) = x e^{-x} g(e^x) \le c
\]
and combining with \(0 \le g(e^{-x}) \le g(1) = 1\) for \(e^x < 2\) and \( x >0\) we deduce
\[
\cA V(q,p) \le \max\{c, \log 2\} + \lambda_\rho \frac{d}{2} - \lambda_\rho \|p\|^2/2.
\]
On $A_2$,
\[
 \|p\|^2 \ge K^2_1 (U(q) + C)
 \]
 so 
\[
\cA V(q, p) \le \left(\max\{c, \log 2\} + \lambda_\rho \frac{d}{2} - \lambda_\rho \frac{C}{2 + K_1^2}\right) - \frac{K^2_1}{2 + K^2_1} \lambda_\rho H(q, p)
\]
and we get
\[
\mathcal A V \le -\alpha_2 V + \beta_2. 
\] 
for some $\alpha_2 > 0, \beta_2 \in \RR$ on $A_2$.
Combining both bounds in an appropriate manner proves the proposition.
\end{proof}

\subsection{Minorization condition}
\begin{proposition}\label{prop:minorization}
    Under \cref{ass:unlocallyconvex}, there exist \(\varsigma\), \(t^*\), and \(\nu\) such that \eqref{eq:minorization} holds.
\end{proposition}
\begin{proof}
We introduce the following kernel corresponding to \(t^* = 3\):
\begin{align}
    \mathbb{K}( a , B) = \mathbb{P}( Z_{t + 3} \in B\;|\; Z_{t}= a).
\end{align}
Recall that $G = \{(q,p) : V(q,p) \leq R\}$ for some $R \ge 2\beta/\alpha$.
Let $\tilde{G} = \{q : \exists p, (q,p) \in G \} $.
Since $G$ is a sub-level set of a quadratically growing function, $G$ is a compact set.
Furthermore, \(\tilde{G}\) is measurable as \(\tilde{G} = \{q : C + U(q) \le R\}\).

One step of the leapfrog integrator is a symplectic transformation, which implies   that for any $q, q_1 \in \Tilde{G}$ we can find a $u \in \mathbb{R}^d$ such that
$(q_1, u_1) = \lf(q, u)$ 
for some $u_1 \in \mathbb{R}^d$. We let $J(q,q_1) = u$. 

Thus, in principle, a refreshment followed by a leapfrog jump followed by a refreshment can go anywhere, and we require a lower bound on this possibility.
As both \(J\) and \(\lf\) are diffeomorphisms, they map compact sets to compact sets, so by compactness there is a uniform upper bound 
\begin{equation*}
    \max_{(q,p) \in G \cup \tilde{G} \times J(\tilde{G}^2) \cup \lf(\tilde{G} \times J(\tilde{G}^2))}\{\lambda_g(q,p),\lambda_\invo(q,p)\} \le c_\lambda.
\end{equation*}
    We divide the time interval $[0,3]$ into three intervals $[0,1]$, $[1,2] $ and $[2, 3]$.
    We consider the following three events:%
    \begin{enumerate}
        \item  The probability of one refreshment happening in each of the intervals $[0,1]$ and $[2,3]$, and no refreshment in $[1,2]$, is given by
        $\lambda_{\rho}^2 e^{-3 \lambda_{\rho}}$.
    \item The probability of no flip, given current state $(q,u)$, happening during $[0,3]$ and no leapfrog jump during $[0,1)\cup(2,3]$ is lower bounded by $e^{-5c_\lambda}$.
    \item The probability of one frog jump, given current state $(q,u)$ (therefore the jump intensity is $\lambda_{g}(q,u)$), during $[1, 1+s)$, $s \in (0, 1]$, no frog jump during $(1+ s, 2]$  can be estimated as 
    \begin{align*}
\int_{0}^{1}&\lambda_{g}(q,u) e^{-\lambda_{g}(q,u)s}   e^{-\lambda_{g}(q_1, u_1)(1-s)} \dd s 
 \\ &\geq e^{-c_{\lambda}} \lambda_{g}(q,u)\int_{0}^{1} e^{-\lambda_{g}(q,u)s  }  \dd s = e^{-c_{\lambda}}\left(1 - e^{-\lambda_{g}(q,u)}\right).
    \end{align*}
    \end{enumerate}
Therefore, using the above estimates, we obtain the following:    
\begin{align}    
\mathbb{K}((q, p)\;, B) \ge \varsigma \nu(B),\quad (q, p) \in G,
\end{align}
where \(\nu(\dd(q_1,p_1)) \propto \ind_{G}(q_1,p_1) e^{-K(p_1)} \dd q_1 \dd p_1\) normalized to a probability measure and
\begin{align*}
\varsigma &= (\tilde\pi \otimes \rho)(G) \lambda_{\rho}^2 e^{-3\lambda_{\rho}} \cdot e^{-5 c_\lambda} \cdot\inf_{q,q_1 \in \tilde{G}} \left(1 - e^{-\lambda_g(q, J(q, q_1))}\right)e^{-c_\lambda} \cdot\inf_{q,q_1 \in \tilde{G} }\frac{e^{-\|J(q, q_1)\|^2/2}}{(2\pi)^{d/2}}.
\end{align*}
The last factor in $\varsigma$ comes from the probability to refresh with $u$.
Due to \cref{lemma:taylor}, we have
\begin{align*}
H(q_1, u_1) - H(q,u) = \frac{\epsilon}{2} \langle u, R_2 \rangle + \frac{\epsilon^2}{8}\|R_2\|^2 + \langle \tilde{\delta}, \tilde{R}_1 - R_2\rangle 
\end{align*}
with $\tilde\delta = \epsilon u - \epsilon^2 \nabla U(q)/2$, $ \tilde{R}_1 = \int_{0}^{1}\nabla^2 U(q + \alpha \tilde \delta)\tilde{\delta} (1 -\alpha) \dd \alpha $ and $R_2$ is the same as in \cref{lemma:taylor}.
Since $J$ is a diffeomorphism, it maps compact sets into compact sets, so
\begin{align*}
    \sup_{(q, q_1) \in \tilde{G}^2} \bigg[\frac{\epsilon}{2} \langle u, R_2 \rangle + \frac{\epsilon^2}{8}\|R_2\|^2 + \langle \tilde{\delta}, \tilde{R}_1 - R_2\rangle  \bigg] < \infty
\end{align*}
and hence $\lambda_{g}(q,u) $ is lower bounded from below. A similar argument ensures a lower bound on $e^{-\|u\|^2/2}$. 
Therefore, $\varsigma$ is lower bounded and \eqref{eq:minorization} follows. 
\end{proof}

\subsection{\cref{thm:geometricergodicity}}
The result now follows by \cite[Corollary~4]{roberts_quantitative_1996} after reformulating our results into their framework.
\begin{proof}[of \cref{thm:geometricergodicity}]
    The formulation differs from \eqref{eq:drift} in that the \(\beta\) term vanishes outside the set considered in the minorization condition.
    However, the conditions on \(R\) in the set \(G\) guarantee that for \(x \notin G\) (i.e., \(V(x) > R \ge 2\beta/\alpha\)) we can rewrite \eqref{eq:drift} as
    \begin{align*}
        \cA V(x) &\le -\alpha V(x) + \beta = -\frac{\alpha}{2}V(x) -\frac{\alpha}{2} \left(R + (V(x) - R)\right) + \beta \\
        &\le -\frac{\alpha}{2}V(x) -\frac{\alpha}{2} R + \beta \le -\frac{\alpha}{2}V(x).
    \end{align*}
    Similarly, for \(x \in G\) we can trivially weaken \eqref{eq:drift} to
    \(\cA V(x) \le -\alpha V(x) + \beta \le -\frac{\alpha}{2} V(x) + \beta\).
    Thus, this other formulation of the drift condition is satisfied as well with a smaller \(\alpha\), therefore \cite[Corollary~4]{roberts_quantitative_1996} is applicable, and the statement follows by \cref{prop:drift,prop:minorization}.
\end{proof}

\section{Rejection scaling of BJS: Proof of \cref{prop:bjslikessqrt}}
\begin{proof}[of \cref{prop:bjslikessqrt}]
Without loss of generality \(\frac{\dd\pi}{\dd\tilde\pi} \neq 0\), so using that the momentum component of \(\pi\) is uniform with respect to \(\tilde\pi\), one obtains
\begin{align*}
    \lambda_\invo(q,p) &= \left(-\int \lambda_g(a,b)(\tilde\mu_\mathrm{step} + \tilde\mu_\mathrm{refl})(a,\dd b) + \int \lambda_g(\invo(a),b)(\tilde\mu_\mathrm{step} + \tilde\mu_\mathrm{refl})(\invo(a),\dd b)\right)^+ \\
    &= \left(-\left[g\left(\frac{\frac{\dd\pi}{\dd\tilde\pi}(q + \epsilon p)}{\frac{\dd\pi}{\dd\tilde\pi}(q)}\right) + \epsilon\langle p, \nabla U(q)\rangle^+\right] + \left[g\left(\frac{\frac{\dd\pi}{\dd\tilde\pi}(q - \epsilon p)}{\frac{\dd\pi}{\dd\tilde\pi}(q)}\right) + \epsilon\langle -p, \nabla U(q)\rangle^+\right]\right)^+\\
    &= \left(-g\left(e^{-(U(q + \epsilon p) - U(q))}\right) + g\left(e^{-(U(q - \epsilon p) - U(q))}\right) - \epsilon\langle p, \nabla U(q)\rangle\right)^+.
\end{align*}
Now expand
\begin{equation*}
    U(q + \epsilon p) - U(q) = \epsilon\langle p, \nabla U(q)\rangle + \frac{\epsilon^2}{2}\langle p, \nabla^2U(q) p\rangle + \cO(\epsilon^3)
\end{equation*}
to obtain that
\begin{align*}
    e^{-(U(q + \epsilon p) - U(q))} &= 1 - \epsilon\langle p, \nabla U(q)\rangle + \frac{\epsilon^2}{2}\left[\langle p, \nabla U(q)\rangle^2 - \langle p, \nabla^2U(q) p\rangle\right] + \cO(\epsilon^3).
\end{align*}
In the case of \(g(t) = \sqrt t\), the above expansion yields cancellation of all terms in \(\lambda_\invo\) of lower order than \(\cO(\epsilon^3)\) as claimed.
One similarly concludes from the preceding analysis that without including the bounce kernel, or with the Metropolized flip rate, the result is \(\cO(\epsilon)\).

It remains to show the scaling in the case of \(g(t) = \min\{1,t\}\), which we do by case analysis.
If \(U(q + \epsilon p) \ge U(q) \ge U(q - \epsilon p)\), then
\begin{align*}
    \lambda_\invo(q,p) &= \left(-\min\left\{1, e^{-(U(q + \epsilon p) - U(q))}\right\} + \min\left\{1, e^{-(U(q - \epsilon p) - U(q))}\right\} - \epsilon\langle p, \nabla U(q)\rangle\right)^+ \\
    &= \left(- e^{-(U(q + \epsilon p) - U(q))} + 1 - \epsilon\langle p, \nabla U(q)\rangle\right)^+ = \cO(\epsilon^2).
\end{align*}
If \(U(q + \epsilon p) \le U(q) \le U(q - \epsilon p)\), the same result is obtained as two signs flip.
Finally, the remaining cases \(U(q + \epsilon p) \le U(q) \ge U(q - \epsilon p)\) and \(U(q + \epsilon p) \ge U(q) \le U(q - \epsilon p)\) both imply that \(\langle p, \nabla U(q)\rangle = 0\); in the former case one has \(\lambda_\invo = 0\), and in the latter case one has the same expansion as for \(g(t)=\sqrt t\) with first order terms vanishing.
Thus the worst case scaling for \(g = \min\{1,t\}\) is indeed \(\cO(\epsilon^2)\) as claimed.
\end{proof}
The preceding also holds when extending to stochastic \(\tilde\mu_\mathrm{refl}\) satisfying the martingale condition, since the integrals then yield the same result.

\clearpage
\section{Additional details on experiments}
\Cref{fig:robustness2d-german-W2} shows the claimed convergence in the Bayesian logistic regression example, with non-convergent configurations eliminated for the later ESS estimation.
\begin{figure}[hbtp]
    \figuresize{.66}
    \figurebox{}{\linewidth}{}[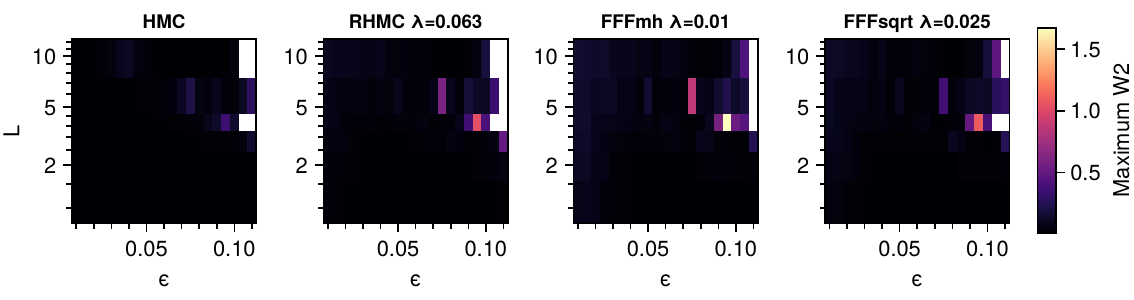]
    \caption{Maximum marginal Wasserstein-2 distance (lower is better) on Bayesian logistic regression target for varying hyperparameters, averaged across replicates.}
    \label{fig:robustness2d-german-W2}
\end{figure}
\begin{figure}[hbtp]
    \figuresize{.66}
    \figurebox{}{\linewidth}{}[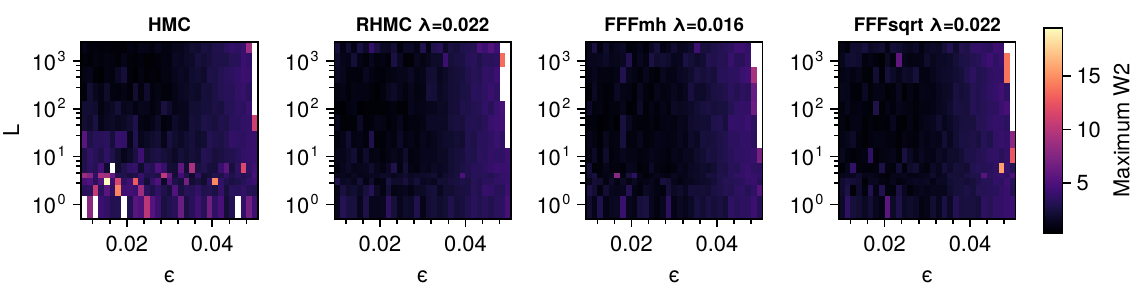]
    \caption{Maximum marginal Wasserstein-2 distance (lower is better) on Rosenbrock banana target for varying hyperparameters, averaged across replicates.}
    \label{fig:robustness2d-banana-W2}
\end{figure}
\begin{figure}[hbtp]%
    \figuresize{.66}
    \figurebox{}{\linewidth}{}[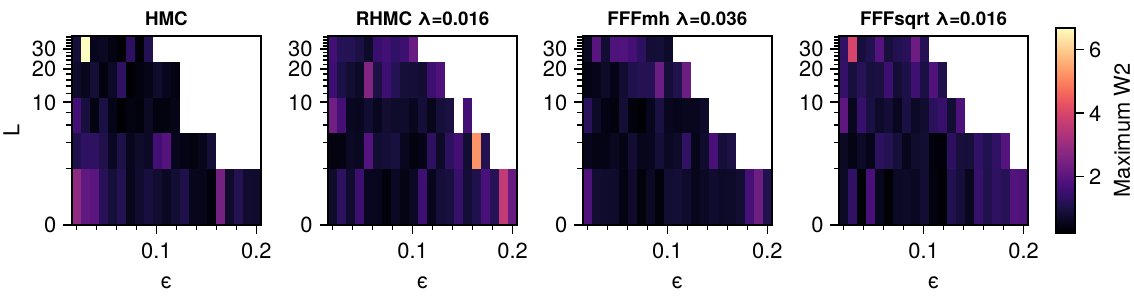]
    \caption{Maximum marginal Wasserstein-2 distance (lower is better) on Bayesian PKPD target for varying hyperparameters, averaged across replicates.}
    \label{fig:robustness2d-posteriordb-W2}
\end{figure}

\Cref{fig:robustness2d-german-ess-lambda,fig:robustness2d-banana-ess-lambda,fig:robustness2d-posteriordb-ess-lambda} shows an alternative slice with fixed \(L\) to illustrate that the robustness is not dependent on a precise \(\lambda_\rho\).
\begin{figure}[hbtp]
    \figuresize{.66}
    \figurebox{}{\linewidth}{}[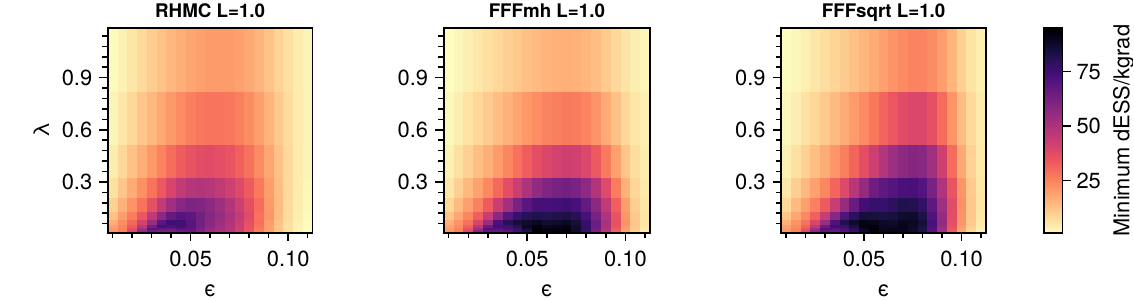]
    \caption{Minimum marginal ESS per \(10^3\) gradient evaluations (higher is better) on Bayesian logistic regression target for varying hyperparameters, averaged across replicates.}
    \label{fig:robustness2d-german-ess-lambda}
\end{figure}
\begin{figure}[hbtp]
    \figuresize{.66}
    \figurebox{}{\linewidth}{}[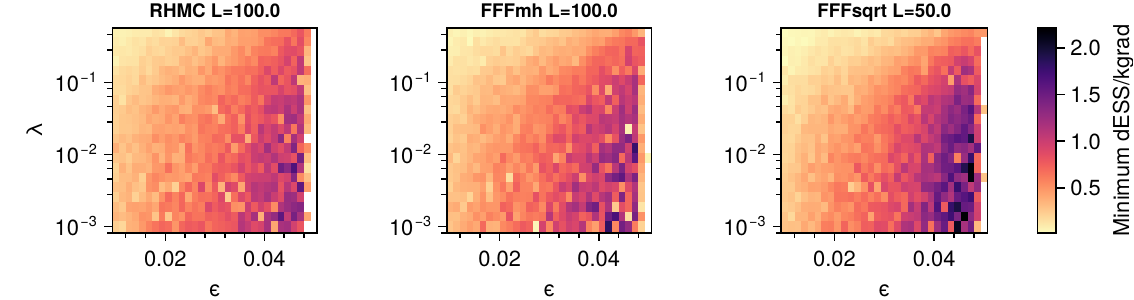]
    \caption{Minimum marginal ESS per \(10^3\) gradient evaluations (higher is better) on Rosenbrock banana target for varying hyperparameters, averaged across replicates.}
    \label{fig:robustness2d-banana-ess-lambda}
\end{figure}
\begin{figure}[hbtp]
    \figuresize{.66}
    \figurebox{}{\linewidth}{}[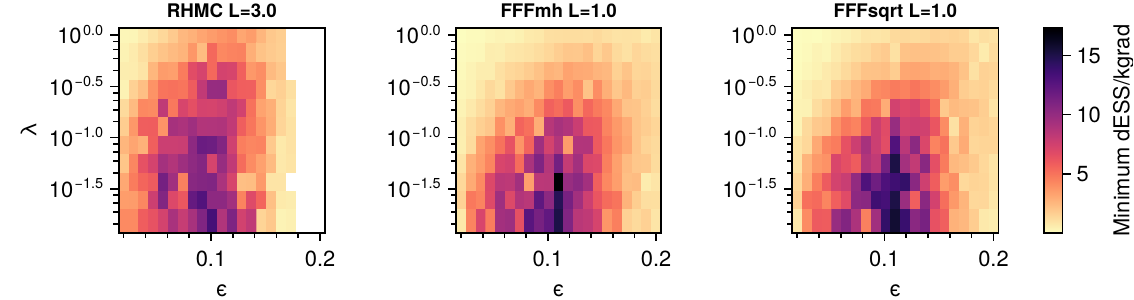]
    \caption{Minimum marginal ESS per \(10^3\) gradient evaluations (higher is better) on Bayesian PKPD target for varying hyperparameters, averaged across replicates.}
    \label{fig:robustness2d-posteriordb-ess-lambda}
\end{figure}

\Cref{tab:robustness-configs} lists the hyperparameter ranges used to produce the figures in the numerical experiments, selected based on preliminary runs. 
\begin{table}[hbtp]
	\tbl{Hyperparameter ranges in numerical experiments.}
{	\begin{tabular}{lllll}
		Experiment & Sampler & \(\epsilon\) & \(L\) & \(\lambda\)\ \\
        Logistic & HMC & \([0.1,0.11]\), 21 linear pts & \(\{1,2,3,4,5,10\}\) & - \\
        Logistic & FFF/RHMC & \([0.1,0.11]\), 21 linear pts & \(\{1,2,3,4,5,10\}\) & \([0.01,1.0]\); 11 log pts \\
        Logistic & BJS & \([0.001, 0.1]\), 31 log pts & - & \([0.001, 1.0]\), 21 log pts\\
		Banana & HMC & \([0.01,0.05]\), 31 linear pts & \(\substack{\{1,2,3,4,5,10,20,50,\\100,200,500,1000,2000\}}\) & - \\
        Banana & FFF/RHMC & \([0.01,0.05]\), 31 linear pts & \(\substack{\{1,2,3,4,5,10,20,50,\\100,200,500,1000,2000\}}\) & \([0.001,0.5]\), 21 log pts \\
        PKPD & HMC & \([0.02,0.2]\), 21 linear pts & \(\{1,3,7,15,31\}\) & - \\
        PKPD & FFF/RHMC & \([0.02,0.2]\), 21 linear pts & \(\{1,3,7,15,31\}\) & \([0.0158,1.0]\), 11 log pts
	\end{tabular}}
	\label{tab:robustness-configs}
\end{table}

\end{document}